\RequirePackage{fix-cm}
\documentclass[smallextended]{svjour3}       
\smartqed  
\usepackage{graphicx,epstopdf}
\usepackage{color}
\graphicspath{{../figs/}}
\begin{document}

\title{Convergence  rates  of  spectral orthogonal projection approximation for functions of
algebraic and  logarithmatic regularities\thanks{This work was supported partly by NSF of China (No.11771454).}}


\author{Shuhuang Xiang}


\institute{Shuhuang Xiang \at
              School of Mathematics and Statistics, Central South University, Changsha, Hunan 410083, P. R. China,
               \email{fxiangsh@mail.csu.edu.cn}           
}

\date{Received: date / Accepted: date}

\maketitle

\begin{abstract}
Based on the Hilb type formula between Jacobi polynomials and Bessel functions,  optimal decay rates on  Jacobi expansion coefficients  are  derived, by applying van der Corput type lemmas, for functions of logarithmatic singularities, which leads to  the optimal convergence rates on the Jacobi, Gegenbauer and Chebyshev orthogonal projections. 
It is interesting to see that for  boundary singularities, one may get faster convergence rate on the Jacobi or  Gegenbauer  projection as  
$(\alpha,\beta)$  and  $\lambda$ increases. The larger values of 
parameter, the higher  convergence rates can be achieved. In particular, the Legendre projection has one half order higher than Chebyshev. Moreover, if $\min\{\alpha,\beta\}>0$ and  $\lambda>\frac{1}{2}$,  
the Jacobi and Gegenbauer  orthogonal projections have higher convergence orders  compared with Legendre. 
While for interior singularity, the convergence order is independent of $(\alpha,\beta)$ and  $\lambda$.

\end{abstract}

\section{Introduction}

The $p$ and $hp$ versions of the finite element methods, or spectral and spectral-element methods  have attracted large interest both in theory and computational practice. 
To deal with corner singularities $(1\pm x)^{\gamma}\ln^{\mu}(1\pm x)$ ($\mu$ a nonnegative integer),  non-uniformly Jacobi-weighted Sobolev spaces,  $H^{m,\beta}(\Omega)$  with integer $m\ge 0$ and $\beta>-1$,
are introduced  to instead of  the standard weighted Sobolev space $H^m_{w}(\Omega)$  with  $w(x)$  a weight on $\Omega$ (\cite{BaGuo1,BaGuo2,BaGuo3,GSW,GW,Hesthaven,LWL,STW}), 
which is applied to  estimate the orthogonal projection
$$
\|f-\mathcal{P}_{N}^{f}\|_{W}\le  \rho(N)\|f\|_{H^{m,\beta}(\Omega)},
$$
where $\mathcal{P}_{N}^{f}(x)$ is the truncated polynomial of the Jacobi (Gegenbauer or Chebyshev)  expansion of $f(x)$, $W$ is a related Sobolev or Besov
space,  $H^{m,\beta}(\Omega)$ defined as a closure of $C^{\infty}$-functions endowed with the weighted norm. The convergence rate $\rho(N)$ depends on the regularity exponentials of  $H^{m,\beta}(\Omega)$  \cite{LWL}.

It is worth noticing that even for $H^{m,\beta}(\Omega)$ with $\Omega=(-1,1)$ and weighted norm  
\begin{equation}
\|u\|_{H^{m,\beta}(\Omega)}=\left\{\sum_{j=0}^m\int_{-1}^1(1-x^2)^{\beta+j}[u^{(j)}(x)]^2dx\right\}^{\frac{1}{2}},
\end{equation}
it could not lead to optimal order  
for $(1+x)^{\gamma}$-type singular functions with non-integer $\alpha>0$ (see \cite[p. 474]{CCS} and  \cite{LWL}). Indeed, for $f(x)=(1+x)^{\gamma}\in  H^{m,-\frac{1}{2}}(\Omega)$ with $m<2\gamma+\frac{1}{2}$,
the Chebyshev approximation $\|f-\mathcal{P}_{N}^{f}\|_{L^2_w(\Omega)}$ with $w(x)=(1-x^2)^{-\frac{1}{2}}$ loss
 an
order of the fractional part of $2\gamma+\frac{1}{2}$, or one order when $2\gamma =k  + \frac{1}{2} $ with  nonnegative integer $k$. 
For more details, see Liu, Wang and Li \cite{LWL}.

To overcome the above deficiency,  Liu, et al. \cite{LWL} introduced a new framework of fractional Sobolev-type
spaces: generalized Gegenbauer
functions of fractional degree (GGF-Fs). Under this framework, the optimal decay rate of Chebyshev expansion coefficients
for a large class of functions with interior and endpoint singularities are presented. In addition,  Hilb type estimates of GGF-Fs are derived in \cite{LW}.
However, for  interior and logarithmic  singularities, it yields
$$
|x-\theta|^s\ln|x-\theta|\in  W_{\theta}^{s+1-\varepsilon}(\Omega),\quad \forall \varepsilon\in (0,1)$$
for $\theta\in \Omega=(-1,1)$, and the Chebyshev expansion coefficients  $c_n$  satisfy
$$
|c_n|=O(n^{s+1-\varepsilon})
$$
(see \cite{LWL}). But the estimate  is still suboptimal. How to modify
the fractional space to best characterize this type singularity  appears
non-trivial and is still open  \cite[Remark 4.5]{LWL}.

For common Gegenbauer expansion, Wang \cite{Wang1}  proposed an alternative derivation of the contour integral representation. With this representation,  optimal
estimates for the Gegenbauer expansion coefficients are derived for analytic functions or functions with endpoint algebraic singularities. However, a precise result on the
asymptotic behavior of  interior singularities or endpoint logarithmic  singularities is still open \cite{Wang1}.

It is of particular  interest to see that for the endpoint singularity $(1\pm x)^{\gamma}\ln^{\mu}(1\pm x)g(x)$ with $g\in C^{\infty}[-1,1]$,   the Jacobi or  Gegenbauer  orthogonal projection can achieve faster convergence rate as  
$(\alpha,\beta)$  and  $\lambda$ increases. The larger values of 
$(\alpha,\beta)$  and  $\lambda$, the higher  convergence rates can be  obtained.
In particular, if $\min\{\alpha,\beta\}>-\frac{1}{2}$ and  $\lambda>0$,  the Jacobi and Gegenbauer  orthogonal projections have higher convergence orders than Chebyshev. While for interior singularity $|x-z_0|^s\ln^{\mu}|x-z_0|g(x)$ ($z_{0}\in(-1,1)$), the convergence order is independent of $(\alpha,\beta)$ and  $\lambda$:
For the Jacobi expansion, it follows that

\begin{equation}\label{JNerror1}
  \|f-\mathcal{P}_{N}^{f}\|_{H^{m,\alpha,\beta}(\Omega)}=\left\{\begin{array}{ll}
   {\cal O}(N^{m-\alpha-2\gamma-1}\ln^{\mu}(N)),& \mbox{$f(x)=(1-x)^{\gamma}\ln^{\mu}(1-x)g(x)$}\\
   {\cal O}(N^{m-s-1/2}\ln^{\mu}(N)),&\mbox{$f(x)=|x-z_0|^s\ln^{\mu}|z-z_0|g(x)$}\\
   {\cal O}(N^{m-\beta-2\delta-1}\ln^{\mu}(N)),& \mbox{$f(x)=(1+x)^{\delta}\ln^{\mu}(1+x)g(x)$}\end{array}\right.
\end{equation}
where $g\in C^{\infty}[-1,1]$,  $\mu$ is a nonnegative integer, $z_{0}\in(-1,1)$, $\min\{\alpha+\gamma,\beta+ \delta,\alpha+2\gamma,\beta+ 2\delta\}>m-1$ for the boundary singularities, 
$s>m -\frac{1}{2}$ and $\min\{\alpha,\beta\}\ge  -\frac{1}{2}$ for  the interior singularity, and
\begin{equation}\label{JN}
\|u\|_{H^{m,\alpha,\beta}(\Omega)}=\left\{\sum_{j=0}^m\int_{-1}^1(1-x)^{\alpha+j}(1+x)^{\beta+j}[u^{(j)}(x)]^2dx\right\}^{\frac{1}{2}}.
\end{equation}
Furthermore, if $\gamma,\, \delta$ are integers and $\mu$ is a positive integer, then 
\begin{equation}\label{JNerror2}
  \|f-\mathcal{P}_{N}^f\|_{H^{m,\alpha,\beta}(\Omega)}=\left\{\begin{array}{ll}
    {\cal O}(N^{m-\alpha-2\gamma-1}\ln^{\mu-1}(N)),& \mbox{$f(x)=(1-x)^{\gamma}\ln^{\mu}(1-x)g(x)$}\\
{\cal O}(N^{m-\beta-2\delta-1}\ln^{\mu-1}(N)),& \mbox{$f(x)=(1+x)^{\delta}\ln^{\mu}(1+x)g(x)$}.\end{array}\right.
\end{equation}
In addition,  the above optimal estimates on the decay rates of the Jacobi expansion can easily lead to 
the  optimal estimates on Gegenbauer expansions for $\alpha=\beta=\lambda-\frac{1}{2}$, and the Chebyshev expansion  for $\alpha=\beta=-\frac{1}{2}$.

These results and the optimal convergence rates can be deduced  from the optimal estimates   on  the decayed rates of the Jacobi expansion coefficients. To 
avoid the  deficiency of the above frameworks and get the optimal asymptotic orders on the expansion coefficients, we will apply van der Corput type lemmas on highly oscillatory Bessel transforms with a large frequency.

Note  that Jacobi polynomial $P_{n}^{(\alpha,\beta)}(x)$ is an oscillatory function around $x=0$, particularly, a highly oscillatory function as $n\gg 1$ (see {\sc Fig. 1.1}).
\begin{figure}[hpbt]\label{fig1}
\centerline{\includegraphics[height=3.5cm,width=4cm]{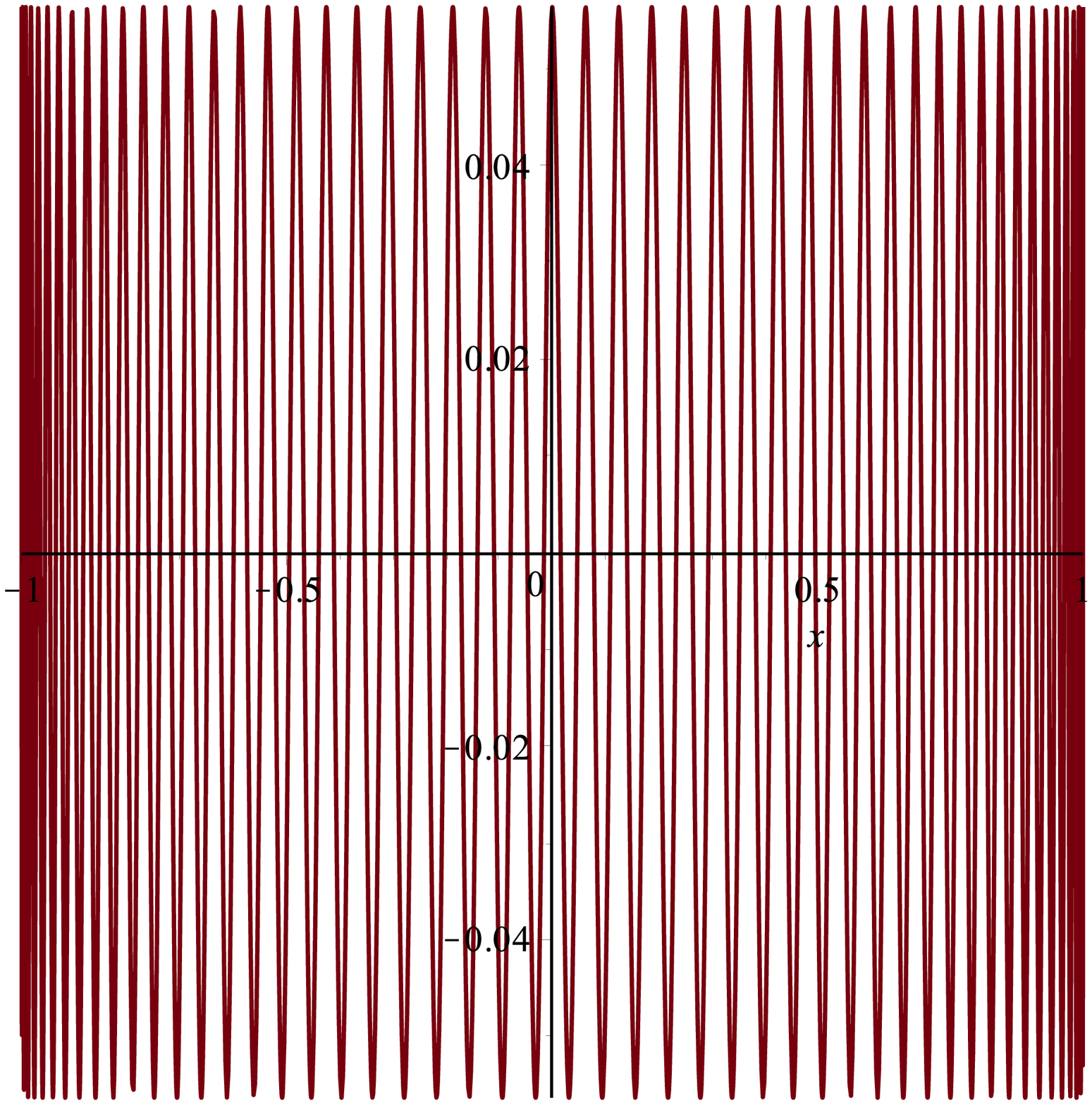}
            \includegraphics[height=3.5cm,width=4cm]{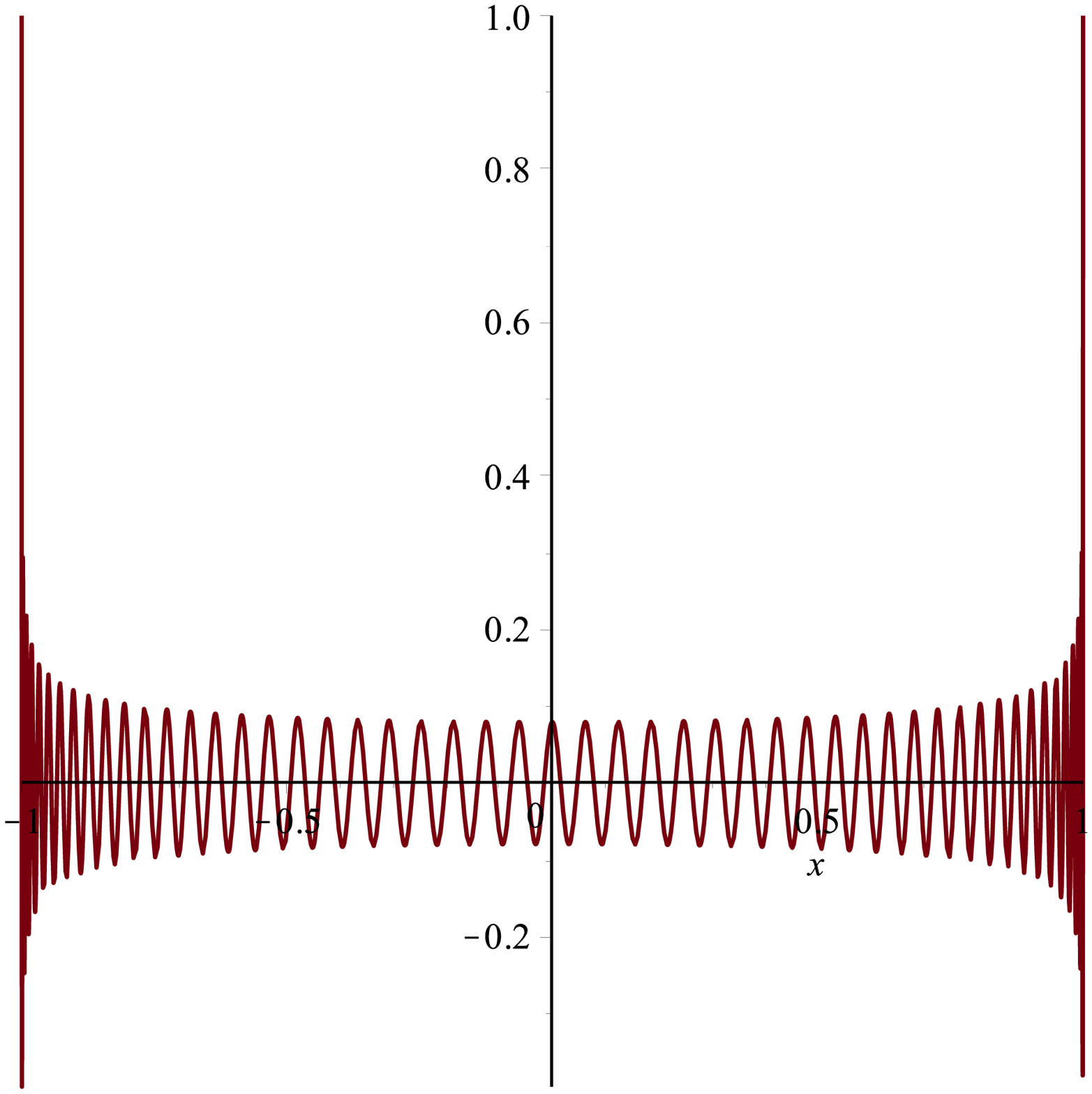}
            \includegraphics[height=3.5cm,width=4cm]{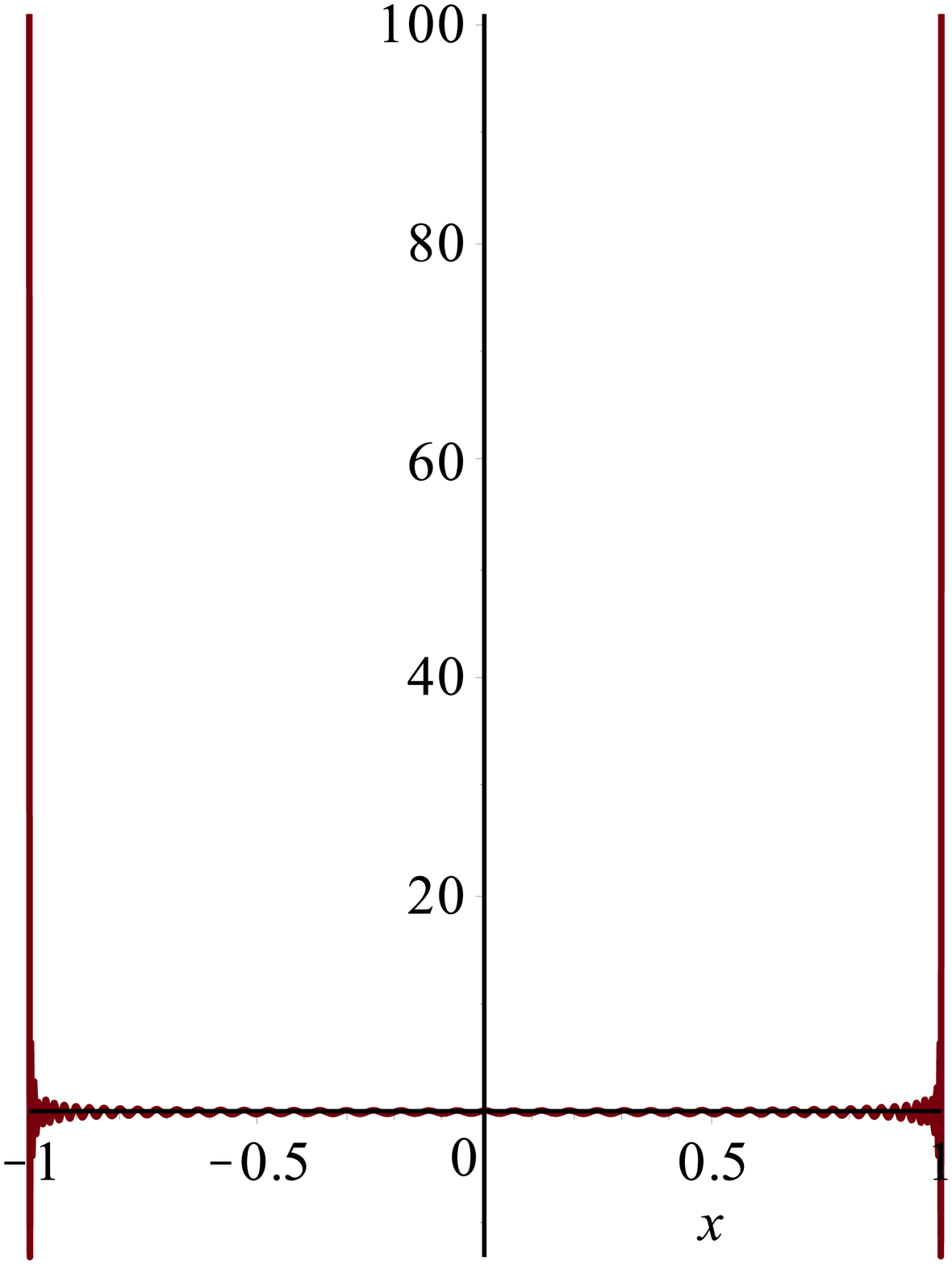}}
\caption{$P_{100}^{(-0.5,-0.5)}(x)$ (left), $P_{100}^{(0,0)}(x)$ (middle), and $P_{100}^{(1,1)}(x)$ (right), respectively.}
\end{figure}
The  Hilb type formula  given in Darboux \cite{Darboux} and Szeg\"{o}  \cite[Theorem 8.21.12]{Szego}
establishes the relation between the Jacobi polynomial and a highly oscillatory Bessel function with a larger frequency when the degree of the Jacobi polynomial  increases.

The paper is organized as follows. In Section 2, we first present the relationship among  the Jacobi,   Gegenbauer and Chebyshev expansion coefficients, 
and derive van der Corput type lemmas for Bessel transforms.  Based on the Hilb type formula, the optimal decay rates on the expansion coefficients for logarithmatic singularities are presented in Section 3, and  the convergence rates on the spectral  orthogonal projections are given in Section 4. Final remarks are included in Section 5.  

\section{Prelimenaries}
Assume $f(x)$ is a suitably smooth function on $[-1,1]$. 
Consider the continuous polynomial expansion
\begin{equation}\label{eq:jacexpan}
  f(x)=\sum_{n=0}^{\infty}a_n(\alpha,\beta)P_n^{(\alpha,\beta)}(x), \quad\quad\alpha,\beta>-1
\end{equation}
with the expansion coefficients
\begin{equation}\label{eq:jacexpcoeffs}
  a_n(\alpha,\beta)=\frac{1}{\sigma_{n}^{\alpha,\beta}}\int_{-1}^{1}\!(1-x)^{\alpha}(1+x)^{\beta}f(x)P_n^{(\alpha,\beta)}(x)\,\mathrm{d}x,
\end{equation}
where $P_n^{(\alpha,\beta)}(x)$ is the Jacobi polynomial of degree $n$ and
\begin{equation}\label{eq:jacsigma}
  \sigma_{n}^{\alpha,\beta}=2^{\alpha+\beta+1}\frac{\Gamma(n+\alpha+1)\Gamma(n+\beta+1)}{n!(2n+\alpha+\beta+1)\Gamma(n+\alpha+\beta+1)}
\end{equation}
(see \cite[p. 774]{Abram}), which leads to the error bound on the orthogonal projection
\begin{equation}\label{truerr}
\|f-\mathcal{P}_{N}^{f}\|_{L^2_w[-1,1]}= \sqrt{\sum_{n=N+1}^{\infty}a_{n}^{2}(\alpha,\beta)\sigma_{n}^{\alpha,\beta}}
\end{equation}
and  implies that  the convergence  is only on the decay of the expansion coefficients. 

In particular,  the optimal estimates on the decay rates of the Jacobi expansion can easily lead to 
the  optimal estimates on Gegenbauer expansion, then derives the estimates on Chebyshev expansion.

The Gegenbauer expansion is
\begin{equation}\label{eq:gegexpan}
  f(x)=\sum_{n=0}^{\infty}a_n(\lambda)C_n^{(\lambda)}(x),\quad\quad\lambda>-\frac{1}{2}, \quad \lambda\not=0
\end{equation}
with the expansion coefficients
$$
  a_n(\lambda)=\frac{1}{\hbar_{n}}\int_{-1}^{1}\!(1-x^2)^{\lambda-\frac{1}{2}}C_n^{(\lambda)}(x)f(x)\,\mathrm{d}x,\quad\hbar_{n}=\frac{2^{1-2\lambda}\pi}{\Gamma^{2}(\lambda)}\frac{\Gamma(n+2\lambda)}{n!(n+\lambda)}
$$
(see \cite[p. 79]{Hesthaven}), and $C_n^{(\lambda)}(x)$ is a Gegenbauer  polynomial of degree $n$
\begin{equation}\label{GenCoe}
  C_n^{(\lambda)}(x)=\frac{\Gamma(\lambda+\frac{1}{2})\Gamma(n+2\lambda)}{\Gamma(2\lambda)\Gamma(n+\lambda+\frac{1}{2})}P_n^{(\lambda-\frac{1}{2},\lambda-\frac{1}{2})}(x), \end{equation}
which is  related to
Legendre and Chebyshev polynomials  as follows \cite[p. 76]{Hesthaven}
\begin{equation}\label{eq:legecheb}
  P_n(x)=C_n^{(\frac{1}{2})}(x), \quad T_n(x)=n\lim_{\lambda\to0}\Gamma(2\lambda)C_n^{(\lambda)}(x).
\end{equation}

From  (\ref{eq:jacexpan})  and  (\ref{GenCoe})
it follows
$$ \begin{array}{lll}
  f(x)&=&{\displaystyle\sum_{n=0}^{\infty}a_{n}\left(\lambda-\frac{1}{2},\lambda-\frac{1}{2}\right)P_n^{(\lambda-\frac{1}{2},\lambda-\frac{1}{2})}(x)}\\
 & =&{\displaystyle \sum_{n=0}^{\infty}a_{n}\left(\lambda-\frac{1}{2},\lambda-\frac{1}{2}\right)\frac{\Gamma(2\lambda)\Gamma(n+\lambda+\frac{1}{2})}{\Gamma(\lambda+\frac{1}{2})\Gamma(n+2\lambda)}C_n^{(\lambda)}(x)}\\
&=&{\displaystyle  \sum_{n=0}^{\infty}a_{n}(\lambda)C_n^{(\lambda)}(x)}\end{array}
  $$
which  derives
\begin{equation}\label{eq:gegcoeffs}
a_{n}(\lambda)=a_{n}\left(\lambda-\frac{1}{2},\lambda-\frac{1}{2}\right)\frac{\Gamma(2\lambda)\Gamma(n+\lambda+\frac{1}{2})}{\Gamma(\lambda+\frac{1}{2})\Gamma(n+2\lambda)}.
\end{equation}

Moreover, note that
$$
  f(x)=\sum_{n=0}^{\infty}a_{n}(\lambda)C_n^{(\lambda)}(x)=\sum_{n=0}^{\infty}\frac{a_{n}(\lambda)}{n\Gamma(2\lambda)}n\Gamma(2\lambda)C_n^{(\lambda)}(x),
$$
which, together with  
\begin{equation}\label{eq:chebexpan}
   f(x)=\sum_{n=0}^{\infty}c_{n}T_n(x)
\end{equation}
and  (\ref{eq:legecheb})-(\ref{eq:gegcoeffs}), derives $c_{0}=a_{n}\left(-\frac{1}{2},-\frac{1}{2}\right)$ and for $n\ge 1$
\begin{equation}\label{eq:checoeffs}\begin{array}{lll}
c_{n}=\displaystyle \lim_{\lambda\to0}\frac{a_{n}(\lambda)}{n\Gamma(2\lambda)}
&=&\displaystyle \lim_{\lambda\to0}\frac{a_{n}\left(\lambda-\frac{1}{2},\lambda-\frac{1}{2}\right)\frac{\Gamma(2\lambda)\Gamma(n+\lambda+\frac{1}{2})}{\Gamma(\lambda+\frac{1}{2})\Gamma(n+2\lambda)}}{n\Gamma(2\lambda)}\\
&=&\displaystyle \frac{\Gamma(n+\frac{1}{2})}{n\Gamma(\frac{1}{2})\Gamma(n)}a_{n}\left(-\frac{1}{2},-\frac{1}{2}\right).
\end{array}
\end{equation}

\vspace{1cm}

The following asymptotic Hilb type formula  for Jacobi polynomials, related to a highly oscillatory Bessel function  with a large frequency, is introduced in Darboux \cite{Darboux} and Szeg\"{o}  \cite[Theorem 8.21.12]{Szego}.
\begin{lemma}\label{lem:lemma1}
  (\cite{Darboux,Szego}) Let $\alpha,\beta>-1$, then as $n\to\infty$
  \begin{equation} \begin{array}{lll}
    & &{\displaystyle
     \theta^{-\frac{1}{2}}\sin^{\alpha+\frac{1}{2}}\left(\frac{\theta}{2}\right)\cos^{\beta+\frac{1}{2}}
     \left(\frac{\theta}{2}\right)P_n^{(\alpha,\beta)}(\cos\theta)}\\
    &=&{\displaystyle\frac{\Gamma(n+\alpha+1)}{\sqrt{2}n!{\tilde N}^{\alpha}}J_{\alpha}({\tilde N}\theta)
     +\left\{\begin{array}{ll}
     \theta^{\frac{1}{2}}{\cal O}\left({\tilde N}^{-\frac{3}{2}}\right),& cn^{-1}\le \theta\le \pi-\epsilon\\
     \theta^{\alpha+2}{\cal O}\left({\tilde N}^{\alpha}\right),& 0< \theta\le cn^{-1},\end{array}\right.}\end{array}
  \end{equation}
  where ${\tilde N}=n+(\alpha+\beta+1)/2$, $c$ and $\epsilon$ are fixed positive numbers, and $J_{\alpha}(z)$ is the first kind of Bessel function of order $\alpha$. The constants in the ${\cal O}$-terms depend on $\alpha$, $\beta$, $c$, and $\epsilon$.
\end{lemma}



 \begin{lemma} 
    Suppose  $\alpha+\nu>0$, $\beta>-1$, $b>0$ and $\mu$ is a nonnegative integer,  then it is satisfied for  $x\in [0,b]$   that  
      \begin{equation}\label{lemma22}
      \quad \Big|\ln^{\mu}(\omega x)(\omega x)^{\alpha}J_{\nu}(\omega x)\Big|=\left\{\begin{array}{ll}
        {\cal O}\left((\ln^{\mu}(\omega))\omega^{\alpha-\frac{1}{2}}\right),&\alpha\ge  \frac{1}{2}\\
         {\cal O}(1),&\alpha< \frac{1}{2},\end{array}\right.\quad \omega \gg 1.
       \end{equation}
   \end{lemma}
\begin{proof}
Define $F_1(z)=\ln^{\mu}(z) z^{\alpha}J_{\nu}(z)$ for $z\in [0,+\infty)$, where $F_1(0)$ is  defined by its limit as $z$ tends to $0$.
From the definition of $J_{\nu}$ (Abramowitz and Stegun
\cite[Eq. 9.1.10]{Abram})
 \begin{equation}\label{lemma133}
J_{\nu}(z)=\left(\frac{z}{2}\right)^{\nu}{\displaystyle\sum_{n=0}^{\infty}\frac{\left(-\frac{1}{4}z^2\right)^n}{n!\Gamma(\nu+n+1)}},
 \end{equation}
we see that  $F_1(0)=0$, $F_1$ is continuous for $z\in [0,+\infty)$. In addition, from \cite[Eq. (9.1.30), Eq. (9.2.1)]{Abram} and  \cite[p. 199]{Watson}
$$
J_{\nu}(z)=\sqrt{\frac{2}{\pi z}}\cos(z-\frac{1}{2}\nu\pi-\frac{\pi}{4})+{\cal O}(z^{-\frac{3}{2}}),\quad z \rightarrow +\infty,
$$
there exists a $z_0\ge 1$ such that for $z\ge z_0$, $|J_{\nu}(z)|\le C_1z^{-\frac{1}{2}}$ for some positive constant $C_1$ independent of $z$, which implies for $\omega x\ge z_0$
$$
|F_1(\omega x)|\le C_1 \ln^{\mu}(\omega x) (\omega x)^{\alpha-\frac{1}{2}},\quad  \alpha\ge \frac{1}{2};\quad 
 |F_1(\omega x)|={\cal O}(1),\quad  \alpha< \frac{1}{2},
$$
then (\ref{lemma22}) is satisfied for  $\omega x\ge z_0$. Notice that $F_1(z)$ is uniformly bounded on $[0,z_0]$, which yields for $\omega x< z_0$
$$
 |F_1(\omega x)|={\cal O}(1).
$$
These together  complete the proof.
\end{proof}

 \begin{lemma} (van der Corput lemma for Bessel transform I)
    Suppose  $\alpha+\nu>-1$, $\beta>-1$, $b>0$, $\mu$ is a nonnegative integer,  $\psi\in
C[0,b]$ and $\psi'\in
L^1[0,b]$, then it is satisfied for  $\omega \gg 1$,         
     \begin{equation}\label{lemma31}
        \int_{0}^{t}\!\ln^{\mu}(x)x^{\alpha}J_{\nu}(\omega x)\,\mathrm{d}x=\left\{\begin{array}{ll}
        {\cal O}\left(\frac{\ln^{\mu}(\omega)}{\omega^{\alpha+1}}\right),&\alpha\le \frac{1}{2}\\
         {\cal O}\left(\omega^{-\frac{3}{2}}\right),&\alpha> \frac{1}{2}\end{array}\right.
      \end{equation}
and
 \begin{equation}\label{lemma32}
        \int_{0}^{t}\!\ln^{\mu}(x)x^{\alpha}\psi(x)J_{\nu}(\omega x)\,\mathrm{d}x=\left\{\begin{array}{ll}
        {\cal O}\left(\frac{\ln^{\mu}(\omega)}{\omega^{\alpha+1}}\right),&\alpha\le  \frac{1}{2}\\
         {\cal O}\left(\omega^{-\frac{3}{2}}\right),&\alpha> \frac{1}{2}\end{array}\right.
      \end{equation}
uniformly  for   $t\in [0,b]$. 
\end{lemma}
\begin{proof}
In the case $\alpha< \frac{1}{2}$: Let $u=\omega x$. It follows
\begin{equation}\label{eq1}
 \int_{0}^{t}\!\ln^{\mu}(x)x^{\alpha}J_{\nu}(\omega x)\,\mathrm{d}x=\frac{1}{\omega^{\alpha+1}}\int_0^{\omega t}
\sum_{j=0}^{\mu}(-1)^jC_{\mu}^j\ln^{\mu-j}(u)\ln^j(\omega)u^{\alpha}J_{\nu}(u)\,\mathrm{d}u
\end{equation}
where $C_{\mu}^j=\frac{\Gamma(\mu+1)}{\Gamma(j+1)\Gamma(\mu-j+1)}$. It is worthy of noting that 
$$
\int_0^{+\infty}u^{\alpha}J_{\nu}(u)du=\left\{\begin{array}{ll}
\frac{
2^{\alpha}\Gamma\left(\frac{\alpha+\nu+1}{2}\right)}{\Gamma\left(\frac{\nu-\alpha+1}{2}\right)}<+\infty,&\nu-\alpha+1\not=0\\
x^{\nu+1}J_{\nu+1}(x)\Big|_0^{+\infty}=0,&\nu-\alpha+1=0\end{array}\right.,
\Re(\alpha+\nu)>-1,\Re(\alpha)<\frac{1}{2}
$$
(\cite[Eq. (9.1.30), Eq. (11.4.16)]{Abram}),  which implies that
$
\int_0^{\omega t}u^{\alpha}J_{\nu}(u)\,\mathrm{d}u
$
is uniformly bounded for $t\in [0,+\infty)$ since this integral with a variable upper bound is continuous and convergent as $t\rightarrow +\infty$.

Then 
from the assumption $\alpha+\nu>-1$, we may choose $\eta$ with  $0<\eta<\min\{\frac{1}{2}-\alpha, \alpha+\nu+1\}$ such that  $\alpha\pm \eta<\frac{1}{2}$, $\alpha+\nu\pm \eta>-1$, which implies for
$j=0,1,\ldots,\mu$ that
 $
\int_0^{1}
[u^{\eta}\ln^{\mu-j}(u)]u^{\alpha-\eta}J_{\nu}(u)\,\mathrm{d}u
$
 is  convergent, and $
\int_1^{+\infty}
[u^{-\eta}\ln^{\mu-j}(u)]u^{\alpha+\eta}J_{\nu}(u)\,\mathrm{d}u
$
is also  convergent by Abel  criterion for improper integrals \cite{Mark} due to that $\int_1^{+\infty}
u^{\alpha+\eta}J_{\nu}(u)\,\mathrm{d}u
$ is convergent and $[u^{-\eta}\ln^{\mu-j}(u)]$ is decreasing and tends to $0$ as $t\rightarrow +\infty$. Thus, $
\int_0^{+\infty}\ln^{\mu-j}(u)u^{\alpha}J_{\nu}(u)\,\mathrm{d}u
$ is convergent
 and $
\int_0^{\omega t}\ln^{\mu-j}(u)u^{\alpha}J_{\nu}(u)\,\mathrm{d}u
$
is uniformly bounded for $t\in [0,+\infty)$ too. These together with  (\ref{eq1}) lead to the desired result (\ref{lemma31}) in the case $\alpha<\frac{1}{2}$.

In the case $\alpha= \frac{1}{2}$:   Noting that $[z^{\nu+1}J_{\nu+1}(z)]'=z^{\nu+1}J_{\nu}(z)$, the integral can be represented for $\mu\ge 1$ as
$$\begin{array}{lll}
 \int_{0}^{t}\!\ln^{\mu}(x)x^{\alpha}J_{\nu}(\omega x)\,\mathrm{d}x
&=& \omega^{-1}\int_{0}^{t}\!\ln^{\mu}(x)x^{\alpha-\nu-1}\,\mathrm{d}[x^{\nu+1}J_{\nu+1}(\omega x)]\\
&=& \omega^{-1}\ln^{\mu}(x)x^{\alpha}J_{\nu+1}(\omega x)\Big|_0^t \\&&- \omega^{-1}\int_{0}^{t}\left[\mu x^{\alpha-1}\ln^{\mu-1}(u)+(\alpha-\nu-1)\ln^{\mu}(x)x^{\alpha-1}\right]J_{\nu+1}(\omega x)\,\mathrm{d}x\\
&=& \omega^{-1-\alpha}[\ln(\omega t)-\ln(\omega)]^{\mu}(\omega t)^{\alpha}J_{\nu+1}(\omega t)\\
&&-\omega^{-1}\int_{0}^{t}\left[\mu x^{\alpha-1}\ln^{\mu-1}(u)+(\alpha-\nu-1)\ln^{\mu}(x)x^{\alpha-1}\right]J_{\nu+1}(\omega x)\,\mathrm{d}x\\
 &=& {\cal O}\left(\ln^{\mu}(\omega))\omega^{-1-\alpha}\right)
  \end{array}
$$
by Lemma 2 and the above  proof since $\alpha-1<\frac{1}{2}$. Similarly, the above estimate is also satisfied  for $\mu=0$.

In the case $\alpha>\frac{1}{2}$:  It yields
$$\begin{array}{lll}
\int_0^t\ln^{\mu}(x)x^{\alpha}J_{\nu}(\omega x)dx
&=&\int_0^t[x^{\alpha-\frac{1}{2}}\ln^{\mu}(x)]x^{\frac{1}{2}}J_{\nu}(\omega x)dx\\
&=&\int_0^t[x^{\alpha-\frac{1}{2}}\ln^{\mu}(x)]d[\int_0^xu^{\frac{1}{2}}J_{\nu}(\omega u)du]\\
&=&[t^{\alpha-\frac{1}{2}}\ln^{\mu}(t)]\int_0^tx^{\frac{1}{2}}J_{\nu}(\omega x)dx- \int_0^t[x^{\alpha-\frac{1}{2}}\ln^{\mu}(x)]' \left(\int_0^xu^{\frac{1}{2}}J_{\nu}(\omega u)du\right)dx\\
&=&{\cal O}\left(\omega^{-\frac{3}{2}}\right)+{\cal O}\left(\omega^{-\frac{3}{2}}\right)\int_0^t[x^{\alpha-\frac{1}{2}}\ln^{\mu}(x)]' dx\\
&=&{\cal O}\left(\omega^{-\frac{3}{2}}\right).\end{array}
$$

 Expression (\ref{lemma32}) follows from (\ref{lemma31}) and 
$$\begin{array}{lll}
\int_0^t\psi (x)\ln^{\mu}(x)x^{\alpha}J_{\nu}(\omega x)dx
&=&\int_0^t\psi (x)\left[\int_0^x\ln^{\mu}(u)u^{\alpha}J_{\nu}(\omega u)du\right]'dx\\
&=&\psi(t)\int_0^t\ln^{\mu}(x)x^{\alpha}J_{\nu}(\omega x)dx-\int_0^t\psi' (x)\left[\int_0^x\ln^{\mu}(u)u^{\alpha}J_{\nu}(\omega u)du\right]dx.\end{array}
$$
\end{proof}

{\sc Remark 1}. From the proof of Lemma 2.3, we see that for  $\alpha+\nu>-1$,
$$
\int_0^t\ln^{\mu}(x)x^{\alpha}J_{\nu}(\omega x)dx={\cal O}\left(\ln^{\mu}(\omega)\omega^{-1-\alpha}\right)
$$
is satisfied uniformly for $t\in [0,+\infty)$  in the case $\alpha<\frac{1}{2}$, while
$$
\int_0^t\ln^{\mu}(x)x^{\frac{1}{2}}J_{\nu}(\omega x)dx={\cal O}\left(\ln^{\mu}(\omega)\omega^{-\frac{3}{2}}\right), \int_0^t\ln^{\mu}(x)x^{\alpha}J_{\nu}(\omega x)dx={\cal O}\left(\omega^{-\frac{3}{2}}\right) \,(\alpha> \frac{1}{2})
$$
uniformly for $t\in [0,b]$. In particular, for $\mu=0$ \cite{Xiang2013}
\begin{equation}\label{Bessel1}
\int_0^tx^{\alpha}J_{\nu}(\omega x)dx={\cal O}\left(\omega^{-\min\{1+\alpha,\frac{3}{2}\}}\right)\left\{ \begin{array}{l}\mbox{uniformly for $t\in [0,+\infty)$  for $\alpha<\frac{1}{2}$}\\
\mbox{uniformly for $t\in [0,b]$ for $\alpha\ge \frac{1}{2}$.}\end{array}\right.\end{equation}

\begin{lemma}  \cite{Xiang2007}  For $0<a<b$,  $\psi\in
C[a,b]$ and $\psi'\in
L^1[a,b]$, 
$$
\int_a^b\psi (t)J_{\nu}(\omega t)dt={\cal O}\left(\omega^{-\frac{3}{2}}\right).
$$
\end{lemma}

 \begin{lemma}\label{lemma5} (van der Corput lemma for Bessel transform II)
    Suppose  $\alpha+\nu>-1$, $\beta>-1$,   $b>0$, $\mu$ is a nonnegative integer, $\psi\in
C[0,b]$ and $\psi'\in
L^1[0,b]$, 
then it is satisfied for  $\omega \gg 1$,         
     \begin{equation}\label{lemma41}
        \int_{0}^{t}\!\ln^{\mu}(x)x^{\alpha}(b-x)^{\beta}\psi(x)J_{\nu}(\omega x)\,\mathrm{d}x=        {\cal O}\left(\max\left\{\frac{\ln^{\mu}(\omega)}{\omega^{\alpha+1}}, \frac{1}{\omega^{\min\{\beta+\frac{3}{2},\frac{3}{2}\}}}\right\}\right)
      \end{equation}
uniformly  for   $t\in [0,b]$. In particular, if $b=1$ and $\mu\ge 1$,  
 \begin{equation}\label{lemma42}
        \int_{0}^{t}\!\ln^{\mu}(x)x^{\alpha}(1-x)^{\beta}\psi(x)J_{\nu}(\omega x)\,\mathrm{d}x=        {\cal O}\left(\max\left\{\frac{\ln^{\mu}(\omega)}{\omega^{\alpha+1}}, \frac{1}{\omega^{\frac{3}{2}}}\right\}\right).
      \end{equation}
      \end{lemma}

      \begin{proof}
        For $t\in [0,\frac{b}{2}]$, from Lemma 3, it establishes
               $$
        \int_{0}^t \ln^{\mu}(x)x^{\alpha}(b-x)^{\beta}J_{\nu}(\omega x)\,\mathrm{d}x =   {\cal O}\left(\max\left\{\frac{\ln^{\mu}(\omega)}{\omega^{\alpha+1}}, \frac{1}{\omega^{\frac{3}{2}}}\right\}\right).
       $$
        
        For $t\in (\frac{b}{2},b]$, the integral can be written as 
  \begin{equation}\label{Besellint}\begin{array}{lll}
      &&  \int_{0}^t \ln^{\mu}(x)x^{\alpha}(b-x)^{\beta}J_{\nu}(\omega x)\,\mathrm{d}x\\&=&\int_{0}^{\frac{b}{2}} \ln^{\mu}(x)x^{\alpha}(b-x)^{\beta}J_{\nu}(\omega x)\,\mathrm{d}x+\int_{\frac{b}{2}}^t \ln^{\mu}(x)x^{\alpha}(b-x)^{\beta}J_{\nu}(\omega x)\,\mathrm{d}x\\
    &=& {\cal O}\left(\max\left\{\frac{\ln^{\mu}(\omega)}{\omega^{\alpha+1}}, \frac{1}{\omega^{\frac{3}{2}}}\right\}\right)+\int_{\frac{b}{2}}^t \ln^{\mu}(x)x^{\alpha}(b-x)^{\beta}J_{\nu}(\omega x)\,\mathrm{d}x.
       \end{array}
   \end{equation}
For the second term on the right-hand side of the first identity (\ref{Besellint}),   setting $F(z)=$ $\int_{\frac{b}{2}}^z\ln^{\mu}(x)x^{\alpha}J_{\nu}(\omega x)dx$, from Lemma 2.4 it yields $F(z)={\cal O}\left( \omega^{-\frac{3}{2}}\right)$.
Moreover, for $t\in [\frac{b}{2}, b-\frac{1}{\omega}]$  it follows that    
  $$ \begin{array}{lll}
        &&\displaystyle \Big|\int_{\frac{b}{2}}^{t}\ln^{\mu}(x)x^{\alpha}(b-x)^{\beta}J_{\nu}(\omega x)\,\mathrm{d}x\Big|\\
         \displaystyle&=&  \Big|\int_{\frac{b}{2}}^{t}(b-x)^{\beta} F'(x)\,\mathrm{d}x\Big|\\
        & \le&\max\{(\frac{b}{2})^{\beta}, |b-t|^{\beta}\}|F(t)|+|\beta|\|F\|_{\infty}\int_{\frac{b}{2}}^{b-\frac{1}{\omega}}(b-x)^{\beta-1}\,\mathrm{d}x\\
         \displaystyle &=&{\cal O}\left(\omega^{-\min\{\beta+\frac{3}{2},\frac{3}{2}\}}\right).\end{array}
        $$
For $t\in (b-\frac{1}{\omega},b]$ it follows
$$ \begin{array}{lll}
&&\int_{\frac{b}{2}}^{t}\ln^{\mu}(x)x^{\alpha}(b-x)^{\beta}J_{\nu}(\omega x)\,\mathrm{d}x\\
&=&\int_{\frac{b}{2}}^{b-\frac{1}{\omega}}\ln^{\mu}(x)x^{\alpha}(b-x)^{\beta}J_{\nu}(\omega x)\,\mathrm{d}x
+\int_{b-\frac{1}{\omega}}^{t}\ln^{\mu}(x)x^{\alpha}(b-x)^{\beta}J_{\nu}(\omega x)\,\mathrm{d}x\\
&=&{\cal O}\left(\omega^{-\min\{\beta+\frac{3}{2},\frac{3}{2}\}}\right)+\int_{b-\frac{1}{\omega}}^{t}\ln^{\mu}(x)x^{\alpha}(b-x)^{\beta}J_{\nu}(\omega x)\,\mathrm{d}x.\end{array}
$$
From 
 $J_{\nu}(\omega z)={\cal O}\left(\omega^{-\frac{1}{2}}\right)$ for $z\in [\frac{b}{2},b]$ (\cite{Abram}), it implies
  $$
\int_{b-\frac{1}{\omega}}^{t}\ln^{\mu}(x)x^{\alpha}(b-x)^{\beta}J_{\nu}(\omega x)\,\mathrm{d}x={\cal O}\left(\omega^{-\frac{1}{2}}\right) \int_{b-\frac{1}{\omega}}^{t}(b-x)^{\beta}\,\mathrm{d}x={\cal O}\left(\omega^{-\beta-\frac{3}{2}}\right).
   $$
These together lead to 
\begin{equation}\label{besselasy11}
        \int_{0}^{t}\!\ln^{\mu}(x)x^{\alpha}(b-x)^{\beta}J_{\nu}(\omega x)\,\mathrm{d}x=        {\cal O}\left(\max\left\{\frac{\ln^{\mu}(\omega)}{\omega^{\alpha+1}}, \frac{1}{\omega^{\min\{\beta+\frac{3}{2},\frac{3}{2}\}}}\right\}\right). 
      \end{equation}

Similarly, Expression (\ref{lemma41}) follows from (\ref{besselasy11}) and 
$$\begin{array}{lll}
&& \int_{0}^t \ln^{\mu}(x)x^{\alpha}(b-x)^{\beta}\psi(x)J_{\nu}(\omega x)\,\mathrm{d}x\\
&=&\int_0^t\psi (x)\left[\int_0^x\ln^{\mu}(u)u^{\alpha}(b-u)^{\beta}J_{\nu}(\omega u)du\right]'dx\\
&=&\psi(t)\int_0^t\ln^{\mu}(x)x^{\alpha}(b-x)^{\beta}J_{\nu}(\omega x)dx-\int_0^t\psi' (x)\left[\int_0^x\ln^{\mu}(u)u^{\alpha}(b-u)^{\beta}J_{\nu}(\omega u)du\right]dx.\end{array}
$$

Particularly, for $b=1$ and $\mu\ge 1$, rewriting
  $$
        \int_{0}^t \ln^{\mu}(x)x^{\alpha}(1-x)^{\beta}\psi(x)J_{\nu}(\omega x)\,\mathrm{d}x 
        = \int_{0}^t [\ln(x)(1-x)^{-1}]^{\mu}x^{\alpha}(1-x)^{\beta+\mu}\psi(x)J_{\nu}(\omega x)\,\mathrm{d}x 
       $$
yields (\ref{lemma42}) due to that  $\beta+\mu+\frac{3}{2}>\frac{3}{2}$.

    \end{proof}

\begin{lemma}\label{lemma6} (Generalized van der Corput lemma \cite{Xiang2019})\label{gcorpt}
  Suppose  $\omega \gg 1$ and $ \phi(x)\in C^{\infty}[0,b]$,  then it is satisfied that
    \begin{equation}\label{Besselasy1}
        \int_{0}^{b}\!x^{\alpha}(b-x)^{\delta-1}\phi(x)J_{\nu}(\omega x)\,\mathrm{d}x=O\left(\omega^{-\min\left\{\alpha+1,\delta+\frac{1}{2}\right\}}\right)
    \end{equation}
  for $\alpha>-1$, $\alpha+\nu>-1$, $\delta>0$.
\end{lemma}

 \begin{lemma}\label{lemma7} (van der Corput lemma for Bessel transform III)
    Suppose  $\alpha+\nu>-1$, $\beta>-1$,   $b>0$, $\mu$ is a nonnegative integer, and $\psi\in
C^{\infty}[0,b]$, then it is satisfied for  $\omega \gg 1$,         
     \begin{equation}\label{la71}
        \int_{0}^{b}\!\ln^{\mu}(x)x^{\alpha}(b-x)^{\beta}\psi(x)J_{\nu}(\omega x)\,\mathrm{d}x=        {\cal O}\left(\max\left\{\frac{\ln^{\mu}(\omega)}{\omega^{\alpha+1}}, \frac{1}{\omega^{\beta+\frac{3}{2}}}\right\}\right).
      \end{equation}
 In particular, if $b=1$ and $\mu\ge 1$,  
 \begin{equation}\label{la72}
        \int_{0}^{1}\!\ln^{\mu}(x)x^{\alpha}(1-x)^{\beta}\psi(x)J_{\nu}(\omega x)\,\mathrm{d}x=        {\cal O}\left(\max\left\{\frac{\ln^{\mu}(\omega)}{\omega^{\alpha+1}}, \frac{1}{\omega^{\beta+\mu+\frac{3}{2}}}\right\}\right).
      \end{equation}
      \end{lemma}

      \begin{proof}
If $\min\left\{\alpha+1,\beta+\frac{3}{2}\right\}\le \frac{3}{2}$, setting $F(x)=    \int_{0}^{x}\!\ln^{\mu}(t)t^{\alpha}(b-t)^{\beta}J_{\nu}(\omega t)\,\mathrm{d}t$, it follows
$$
    \int_{0}^{b}\!\ln^{\mu}(x)x^{\alpha}(b-x)^{\beta}\psi(x)J_{\nu}(\omega x)\,\mathrm{d}x=    \int_{0}^{b}\psi(x)\,\mathrm{d}F(x)
    =\psi(b)F(b) -  \int_{0}^{b}\psi'(x)F(x)\mathrm{d}x,
$$
which implies the desired result by Lemma 5.

 If $ \frac{3}{2}<\min\left\{\alpha+1,\beta+\frac{3}{2}\right\}\le \frac{5}{2}$,    by integrating by parts, it follows for $\mu\ge 1$ that
 $$\begin{array}{lll}
  && \int_{0}^{b}\!\ln^{\mu}(x)x^{\alpha}(b-x)^{\beta}\psi(x)J_{\nu}(\omega x)\,\mathrm{d}x\\
&=& \omega^{-1}\int_{0}^{b}\ln^{\mu}(x)x^{\alpha-\nu-1}(b-x)^{\beta}\psi(x)\,\mathrm{d}[x^{\nu+1}J_{\nu+1}(\omega x)]\\
&=&- \omega^{-1}\int_{0}^{b}
\ln^{\mu}(x)x^{\alpha-1}(b-x)^{\beta-1}\phi(x)J_{\nu+1}(\omega x)\,\mathrm{d}x\\
&&- \mu \omega^{-1}\int_{0}^{b}
x^{\alpha-1}(b-x)^{\beta-1}\ln^{\mu-1}(x)(b-x)\psi(x)J_{\nu+1}(\omega x)\,\mathrm{d}x\\    
&=&  {\cal O}\left(\max\left\{\frac{\ln^{\mu}(\omega)}{\omega^{\alpha+1}}, \frac{1}{\omega^{\beta+\frac{3}{2}}}\right\}\right),\end{array}
$$
where $\phi(x)=(\alpha-\nu-1)(b-x)\psi(x)-\beta x\psi(x)+x(b-x)\psi'(x)$. In the case $\mu=0$, it directly follows from Lemma 6.

   By induction on  $k$ for $k+\frac{3}{2}<\min\left\{\alpha+1,\beta+\frac{3}{2}\right\}\le k+\frac{5}{2}$ with $k=0,1,\ldots$,  by a similar proof to the above, it is easy to derive the desired result (\ref{la71}).
           
  The special case $b=1$ and $\mu\ge 1$ follows from (\ref{la71}) together with
  $$
        \int_{0}^1 \ln^{\mu}(x)x^{\alpha}(1-x)^{\beta}\psi(x)J_{\nu}(\omega x)\,\mathrm{d}x 
        = \int_{0}^1 [\ln(x)(1-x)^{-1}]^{\mu}x^{\alpha}(1-x)^{\beta+\mu}\psi(x)J_{\nu}(\omega x)\,\mathrm{d}x. 
       $$

           \end{proof}

By a similar proof, the following are satisfied.

 \begin{lemma}\label{lemma8} (van der Corput lemma for Bessel transform IV)
    Suppose  $\alpha+\nu>-1$, $\beta>-1$,   $b>0$, $\mu$ is a nonnegative integer, $\psi\in
C[0,b]$ and $\psi'\in
L^1[0,b]$, 
then it is satisfied for  $\omega \gg 1$,         
     \begin{equation}\label{la8}
        \int_{0}^{t}\!\ln^{\mu}(b-x)x^{\alpha}(b-x)^{\beta}\psi(x)J_{\nu}(\omega x)\,\mathrm{d}x=        {\cal O}\left(\max\left\{\frac{1}{\omega^{\alpha+1}}, \frac{\ln^{\mu}(\omega)}{\omega^{\min\{\beta+\frac{3}{2},\frac{3}{2}\}}}\right\}\right). 
      \end{equation}
uniformly  for   $t\in [0,b]$.     In particular, if $b=1$ and $\mu\ge 1$,  $\alpha$ in (\ref{la8}) can be replaced by $\alpha+\mu$.
 \end{lemma}

 \begin{lemma}\label{lemma9} (van der Corput lemma for Bessel transform V)
    Suppose  $\alpha+\nu>-1$, $\beta>-1$,   $b>0$, $\mu$ is a nonnegative integer, and $\psi\in
C^{\infty}[0,b]$, then it is satisfied for  $\omega \gg 1$,         
     \begin{equation}\label{la9}
        \int_{0}^{b}\!\ln^{\mu}(b-x)x^{\alpha}(b-x)^{\beta}\psi(x)J_{\nu}(\omega x)\,\mathrm{d}x=        {\cal O}\left(\max\left\{\frac{1}{\omega^{\alpha+1}}, \frac{\ln^{\mu}(\omega)}{\omega^{\beta+\frac{3}{2}}}\right\}\right).
      \end{equation}
        In particular, if $b=1$ and $\mu\ge 1$,  $\alpha$ in (\ref{la9}) can be replaced by $\alpha+\mu$.
      \end{lemma}

The asymptotics on Lemma \ref{lemma7} and Lemma \ref{lemma9} are illustrated with different values of $(\alpha,\beta)$ by  $\int_0^b\ln^{\mu}(x)x^{\alpha}(b-x)^{\beta}\psi(x)J_{\nu}(\omega x)dx$ with $b=\frac{1}{2}$, $\mu=1$ and $\psi(x)=\cos(x)$ for $\omega=1:1000$ in {\sc fig. } \ref{fig:figs2.1} which shows the estimates are tight.

\begin{figure}[hpbt]
\centerline{\includegraphics[height=6cm,width=16cm]{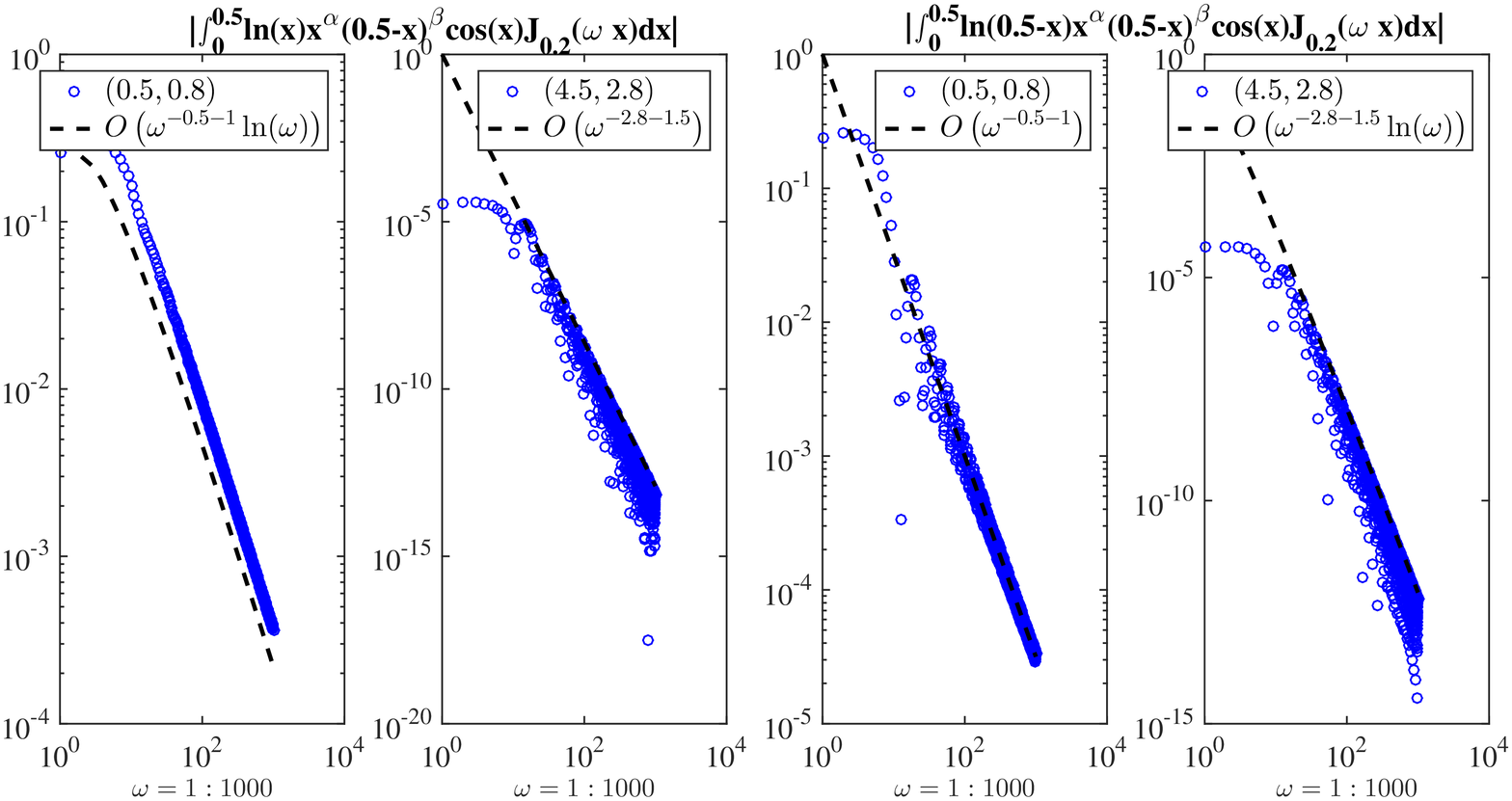}}
\caption{The asymptotics on Lemma 7 and Lemma 9 are illustrated with different values of $(\alpha,\beta)$.}\label{fig:figs2.1}
\end{figure}


\section{On functions of limited regularity at endpoints or interior points}

\subsection{Functions with boundary regularities}
Now we consider
\begin{equation}\label{rend}
  f(x)=(1-x)^{\gamma}\ln^{\mu}(1-x)g(x),
\end{equation}
where  $\mu$ is a positive integer  and $g\in C^{\infty}[-1,1]$.

\begin{theorem}
Suppose $f(x)$  is defined by (\ref{rend}), then the Jacobi coefficients (\ref{eq:jacexpcoeffs})  satisfy  for $\alpha+\gamma>-1$ that
\begin{equation}\label{Jacobiend1}
  |a_{n}(\alpha,\beta)|=\left\{\begin{array}{ll}
  {\cal O}\left(n^{-\alpha-2\gamma-1}\ln^{\mu}(n)\right),&\mbox{$\gamma$  is not  an integer }\\
  {\cal O}\left(n^{-\alpha-2\gamma-1}\ln^{\mu-1}(n)\right),&\mbox{$\gamma$ is an  integer}.\end{array}\right.
\end{equation}
\end{theorem}

\begin{proof}
Let  $k_0$ be a positive integer such that $\alpha+2\gamma-k_0+2\le \frac{3}{2}\le \beta+k_0+2$. Then $f^{(k_0)}(x)$ can be represented for $\gamma> 0$  not  an integer as 
$$\begin{array}{ll}
f^{(k_0)}(x)&=(1-x)^{\gamma-k_0}\left[\ln^{\mu}(1-x)h_{\mu}(x)+\ln^{\mu-1}(1-x)h_{\mu-1}(x)\right.\\
&\quad \left.+\cdots+\ln(1-x)h_{1}(x)+h_0(x)\right]\\
&=:(1-x)^{\gamma-k_0}\psi(x)\end{array}
$$
with $h_j\in C^{\infty}[-1,1],\, j=0,1,\ldots,\mu$.

Moreover, by  Rodrigues' formula  \cite[p. 94, (4.10.1)]{Szego}
$$ \begin{array}{lll}&&{\displaystyle
(1-x)^{\alpha}(1+x)^{\beta}P_n^{(\alpha,\beta)}(x)}\\
&=&{\displaystyle\frac{(-1)^k}{2^kn(n-1)\cdots(n-k+1)}\frac{\mathrm{d}^k}{\mathrm{d}x^k}\left\{(1-x)^{k+\alpha}(1+x)^{k+\beta}P_{n-k}^{(k+\alpha,k+\beta)}(x)\right\}},\end{array}
$$
we have
\begin{equation}\label{JacInt}
\begin{array}{ll}
  a_{n}(\alpha,\beta)&\displaystyle=\frac{1}{\sigma_{n}^{\alpha,\beta}}\int_{-1}^{1}\!(1-x)^{\alpha}(1+x)^{\beta}f(x)P_n^{(\alpha,\beta)}(x)\,\mathrm{d}x\\
 & \displaystyle=\frac{\int_{-1}^{1}\!(1-x)^{\alpha+k_0}(1+x)^{\beta+k_0}P_{n-k_0}^{(\alpha+k_0,\beta+k_0)}(x)
  f^{(k_0)}(x)\,\mathrm{d}x}{(-1)^{k_0}2^{k_0}\sigma_{n}^{\alpha,\beta}n(n-1)\cdots(n-k_0+1)}\\
  &\displaystyle=\frac{\int_{-1}^{1}\!(1-x)^{\alpha+\gamma}(1+x)^{\beta+k_0}P_{n-k_0}^{(\alpha+k_0,\beta+k_0)}(x)\psi(x)\,\mathrm{d}x}{(-1)^{k_0}2^{k_0}\sigma_{n}^{\alpha,\beta}n(n-1)\cdots(n-k_0+1)}.\\
  \end{array}
\end{equation}
Setting $x=\cos\theta$ and noting that $\ln(1-x)=2\ln(\theta)+2\ln\frac{\sin\theta/2}{\theta/2}-\ln2$ and $\ln\frac{\sin\theta/2}{\theta/2}\in C^{\infty}[0,\pi]$, then $\psi(x)$ can be represented as
$$
\psi(\cos\theta)=\ln^{\mu}(\theta){\hat h}_{\mu}(\cos\theta)+\ln^{\mu-1}(\theta){\hat h}_{\mu-1}(\cos\theta)+\cdots+\ln(\theta){\hat h}_{1}(\cos\theta)+{\hat h}_{0}(\cos\theta)
$$
with ${\hat h}_j\in C^{\infty}[0,\pi],\, j=0,1,\ldots,\mu$. Then
by Lemma 1, Lemma 3 and (\ref{Bessel1}), the numerator in (\ref{JacInt}) can be estimated as
$$\begin{array}{ll}
&{\displaystyle \int_{-1}^{1}\!(1-x)^{\alpha+\gamma}(1+x)^{\beta+k_0}P_{n-k_0}^{(\alpha+k_0,\beta+k_0)}(x)\psi(x)\,\mathrm{d}x}\\
=&\displaystyle2^{\alpha+\beta+\gamma+k_{0}+1}\int_{0}^{\pi}\sin^{2\alpha+2\gamma+1}\frac{\theta}{2}\cos^{2\beta+2k_0+1}\frac{\theta}{2}P_{n-k_0}^{(\alpha+k_0,\beta+k_0)}(\cos\theta)\psi(\cos\theta)\,\mathrm{d}\theta\\
=&\displaystyle2^{\alpha+\beta+\gamma+k_{0}+1}\int_{0}^{\pi/2}\sin^{2\alpha+2\gamma+1}\frac{\theta}{2}\cos^{2\beta+2k_0+1}\frac{\theta}{2}P_{n-k_0}^{(\alpha+k_0,\beta+k_0)}(\cos\theta)\psi(\cos\theta)\,\mathrm{d}\theta\\
&+\displaystyle2^{\alpha+\beta+\gamma+k_{0}+1}\int_{\pi/2}^{\pi}\sin^{2\alpha+2\gamma+1}\frac{\theta}{2}\cos^{2\beta+2k_0+1}\frac{\theta}{2}P_{n-k_0}^{(\alpha+k_0,\beta+k_0)}(\cos\theta)\psi(\cos\theta)\,\mathrm{d}\theta\\
=&\displaystyle\frac{\Gamma(n+\alpha+1)}{(n-k_{0})!{\tilde N}^{\alpha+k_{0}}}
\int_{0}^{\pi/2}\!\theta^{\alpha+2\gamma-k_0+1} [\ln^{\mu}(\theta){\tilde h}_{\mu}(\theta)+\cdots+\ln(\theta){\tilde h}_{1}(\theta)+{\tilde h}_{0}(\theta)]J_{\alpha+k_0}({\tilde N}\theta)\,\mathrm{d}\theta\\
&+{\cal O}({\tilde N}^{-\frac{3}{2}})+\displaystyle\frac{(-1)^{n-k_0}\Gamma(n+\beta+1)}{(n-k_{0})!{\tilde N}^{\beta+k_{0}}}
\int_{0}^{\pi/2}\!\theta^{\beta+k_0+1}
J_{\beta+k_0}({\tilde N}\theta)\hat{\psi}_1(\theta)\,\mathrm{d}\theta+{\cal O}({\tilde N}^{-\frac{3}{2}})\\
=&\displaystyle {\cal O}\left({\tilde N}^{-\min\{\alpha+2\gamma-k_0+2,\frac{3}{2}\}}\ln^{\mu}({\tilde N})\right)+{\cal O}\left({\tilde N}^{-\min\{\alpha+2\gamma-k_0+2,\frac{3}{2}\}}\ln^{\mu-1}({\tilde N})\right)\\
&\displaystyle  +\cdots+{\cal O}\left({\tilde N}^{-\min\{\alpha+2\gamma-k_0+2,\frac{3}{2}\}}\right)+{\cal O}\left({\tilde N}^{-\min\{\beta+k_0+2,\frac{3}{2}\}}\right)+{\cal O}({\tilde N}^{-\frac{3}{2}})\\
=&\displaystyle {\cal O}\left({\tilde N}^{-\alpha-2\gamma+k_0-2}\ln^{\mu}({\tilde N})\right)
\end{array}
$$
where
$$\left\{\begin{array}{l}
{\tilde h}_{j}(\theta)=2^{\alpha+\beta+\gamma+k_{0}+\frac{1}{2}}\left(\frac{\sin\theta/2}{\theta}\right)^{\alpha+2\gamma-k_{0}+\frac{1}{2}}\cos^{\beta+k_0+\frac{1}{2}}\left(\frac{\theta}{2}\right){\hat h}_{j}(\theta),\\
\hat{\psi}_1(\theta)=2^{\alpha+\beta+\gamma+k_{0}+\frac{1}{2}}\left(\frac{\sin\theta/2}{\theta}\right)^{\beta+k_{0}+\frac{1}{2}}\cos^{\alpha+2\gamma-k_0+\frac{1}{2}}\left(\frac{\theta}{2}\right)\psi(-\cos(\theta)), \end{array}\right.
$$
and  ${\tilde N}=n+(\alpha+\beta+1)/2$. It is easy to verify that  $\hat{\psi}_1(\theta),\,{\tilde h}_{j}(\theta)\in C^{\infty}[0,\frac{\pi}{2}]$ for $j=0,\ldots,\mu$, then  together with $\sigma_{n}^{\alpha,\beta}={\cal O}(n^{-1})$,  it derives
$$\begin{array}{ll}
a_{n}(\alpha,\beta) &\displaystyle=\frac{\int_{-1}^{1}\!(1-x)^{\alpha+\gamma}(1+x)^{\beta+k_0}P_{n-k_0}^{(\alpha+k_0,\beta+k_0)}(x)\psi(x)\,\mathrm{d}x}{(-1)^{k_0}2^{k_0}\sigma_{n}^{\alpha,\beta}n(n-1)\cdots(n-k_0+1)}\\
&={\cal O}\left(n^{-\alpha-2\gamma-1}\ln^{\mu}(n)\right)  \end{array}.
$$

In the case $\gamma$ is a nonnegative  integer and $\mu>1$, we select $k_0$  further satisfied $k_0>\max\{\gamma,\mu\}$, then
$f^{(k_0)}(x)$ can be represented  as 
$$
f^{(k_0)}(x)=\sum_{j=0}^{k_0}(1-x)^{\gamma-k_0+j}\sum_{i=0}^{\mu-1}\ln^{i}(1-x)\phi_{j,i}(x)+(1-x)^{\gamma}\ln^{\mu}(x)\phi_0(x)
$$
with $\phi_0\in C^{\infty}[-1,1]$ and $\phi_{j,i}\in C^{\infty}[-1,1]$.
In analogy to the above proof, it also leads to the desired result.

In the case $\gamma$ is a nonnegative  integer and $\mu=1$,
$f^{(k_0)}(x)$ can be represented  as 
$$
f^{(k_0)}(x)=\ln(1-x)\psi_1(x)+(1-x)^{\gamma-k_0}\psi_2(x) \quad \psi_1,\psi_2\in C^{\infty}[-1,1].
$$
In analogy to the above proof, it leads to 
$$
\int_{-1}^{1}\!(1-x)^{\alpha+k_0}(1+x)^{\beta+k_0}P_{n-k_0}^{(\alpha+k_0,\beta+k_0)}(x)\ln(1-x)\psi_1(x)\,\mathrm{d}x={\cal O}(n^{-\frac{3}{2}})
$$
by $\alpha+k_0+1>\frac{1}{2}$, and
$$
\int_{-1}^{1}\!(1-x)^{\alpha+\gamma}(1+x)^{\beta+k_0}P_{n-k_0}^{(\alpha+k_0,\beta+k_0)}(x)\psi_2(x)\,\mathrm{d}x={\cal O}\left(n^{-\alpha-2\gamma+k_0-2}\right)  
$$
which together deduces the desired result (\ref{Jacobiend1}).
\end{proof}

Similar results can be obtained for 
\begin{equation}\label{lend}
  f(x)=(1+x)^{\delta}\ln^{\mu}(1+x)g(x),
\end{equation}
where  $\mu$ is a positive integer  and $g\in C^{\infty}[-1,1]$.

\begin{theorem}
Suppose $f(x)$ is defined by (\ref{lend}), then the Jacobi coefficients (\ref{eq:jacexpcoeffs})  satisfy that for $\beta+\delta>-1$
\begin{equation}\label{Jacobiend2}
  |a_{n}(\alpha,\beta)|=\left\{\begin{array}{ll}
  {\cal O}\left(n^{-\beta-2\delta-1}\ln^{\mu}(n)\right),&\mbox{$\delta$  is not  an integer }\\
  {\cal O}\left(n^{-\beta-2\delta-1}\ln^{\mu-1}(n)\right),&\mbox{$\delta$ is an  integer}.\end{array}\right.
\end{equation}
\end{theorem}

Now we consider
\begin{equation}\label{eq:boundsingularfunc}
  f(x)=(1-x)^\gamma(1+x)^{\delta}\ln^{\mu}(1-x^2)g(x),
\end{equation}
where  $\mu$ is a positive integer and $g\in C^{\infty}[-1,1]$. Then from Tuan and Elloiit \cite{Tuan}, $f(x)$ can be rewritten as
$$
  f(x)=(1-x)^{\gamma}\ln^{\mu}(1-x)g_1(x)+(1+x)^{\delta}\ln^{\mu}(1+x)g_2(x)
$$
with $g_1,\, g_2 \in C^{\infty}[-1,1]$.

\begin{corollary}
Suppose $f(x)$ is defined as (\ref{eq:boundsingularfunc}), then the Jacobi coefficients (\ref{eq:jacexpcoeffs})  satisfy that
\begin{equation}\label{Jacobiend}
   |a_{n}(\alpha,\beta)|=\left\{\begin{array}{ll}
      {\cal O}\left(n^{-\min\left\{\alpha+2\gamma+1,\beta+2\delta+1\right\}}\ln^{\mu-1}(n)\right),&\gamma,\delta \in{\cal N}_0\\
      {\cal O}\left(\min\left\{n^{-\alpha-2\gamma-1}\ln^{\mu-1}(n),n^{-\beta-2\delta-1}\ln^{\mu}(n)\right\}\right),&\gamma \in
      {\cal N}_0,\ \delta\notin{\cal N}_0\\
       {\cal O}\left(\min\left\{n^{-\beta-2\delta-1}\ln^{\mu}(n),n^{-\alpha-2\gamma-1}\ln^{\mu-1}(n)\right\}\right),&\gamma \notin
      {\cal N}_0,\ \delta\in{\cal N}_0\\
       {\cal O}\left(n^{-\min\left\{\alpha+2\gamma+1,\beta+2\delta+1\right\}}\ln^{\mu}(n)\right),&\gamma \notin
      {\cal N}_0,\ \delta\notin{\cal N}_0\\
      \end{array}\right.
\end{equation}
\end{corollary}
where $\min\{\alpha+\gamma,\beta+\delta\}>-1$ and ${\cal N}_0$ is the set of  integers.

We illustrate the  decay rates  by $f(x)=(1-x)^{\gamma}\ln(1-x)$ with different values of
$\alpha$ and $\beta$ (see {\sc Fig.} \ref{fig:figs2}). These numerical results are in accordance with the estimates.  The  asymptotic orders of the decay on the coefficients are sharp.

\begin{figure}[hpbt]
      \centerline{\includegraphics[height=5cm,width=16cm]{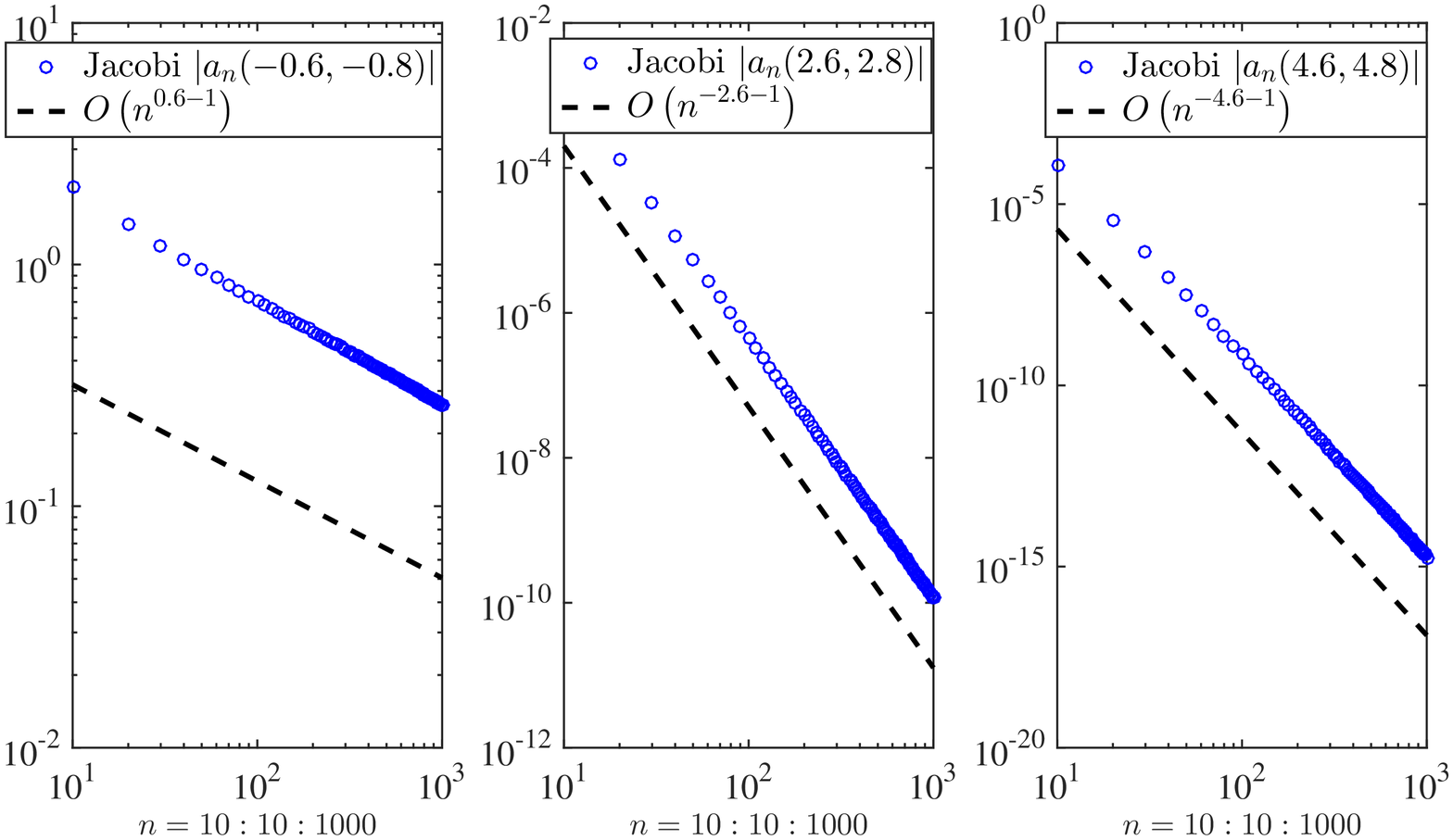}}
       \centerline{\includegraphics[height=5cm,width=16cm]{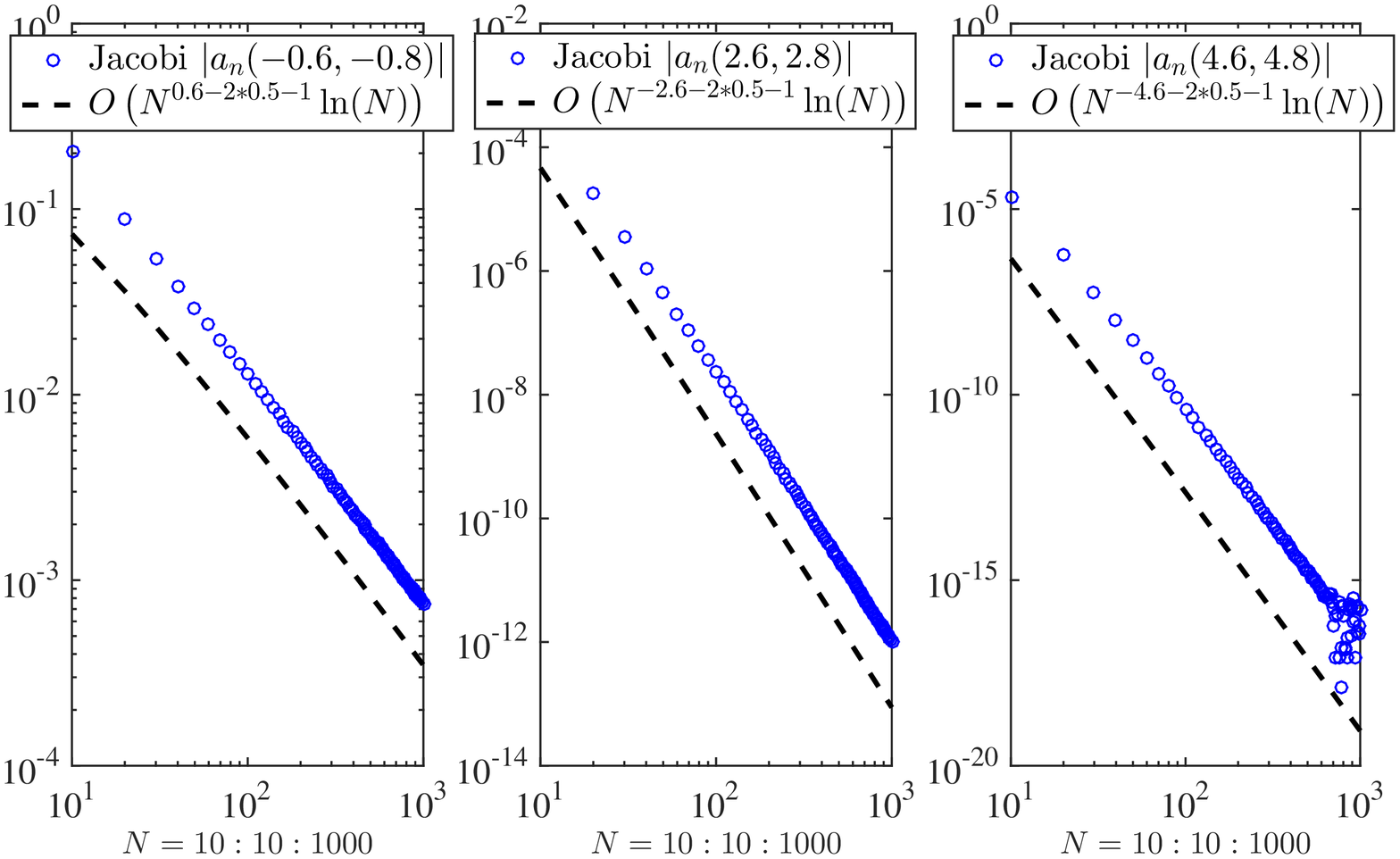}}
 \centerline{\includegraphics[height=5cm,width=16cm]{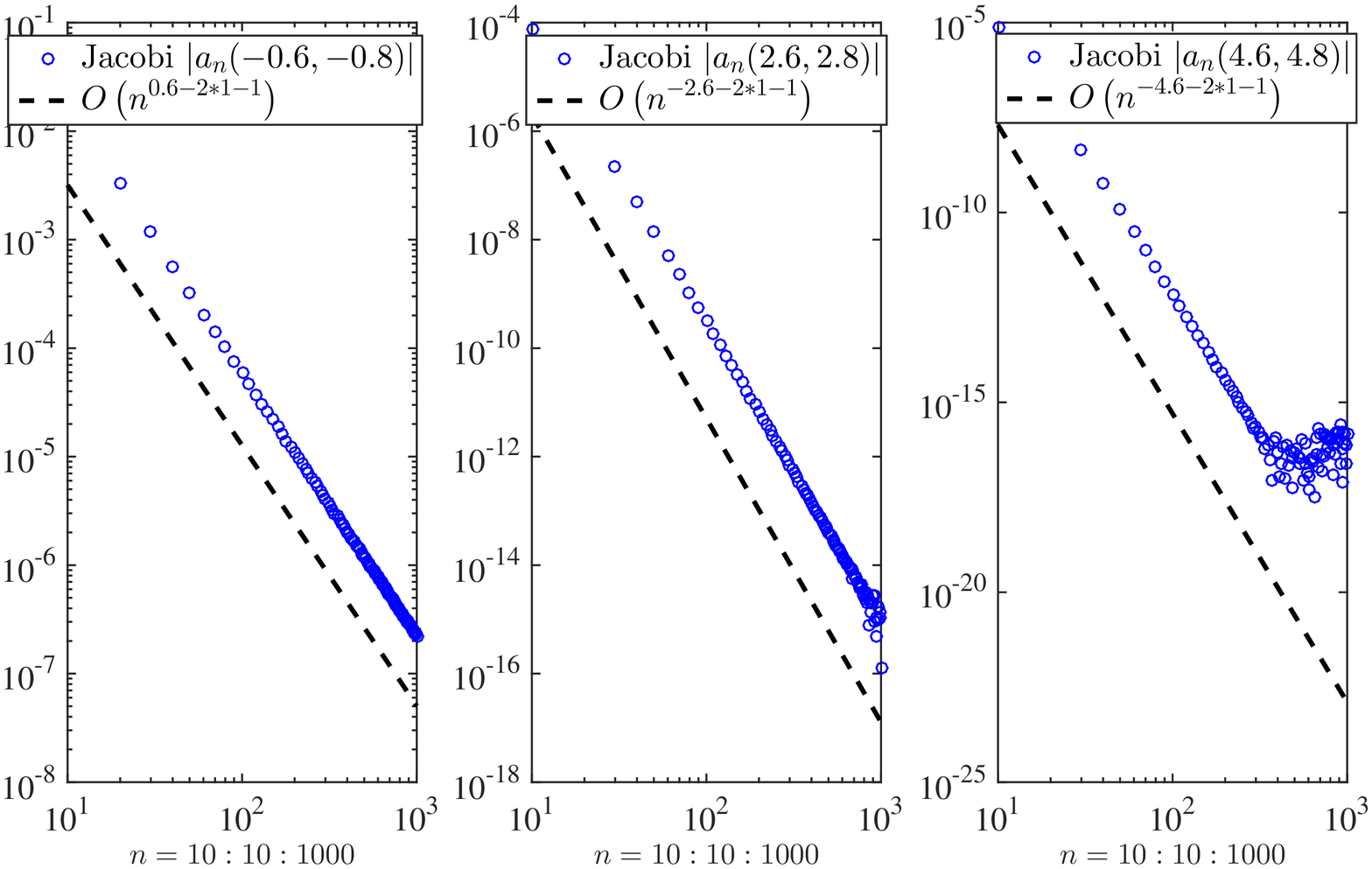}}
        \caption{The asymptotic decay of the Jacobi coefficients $|a_{n}(\alpha,\beta)|$ in (\ref{eq:jacexpan}) for  $f(x)=(1-x)^{\gamma}\ln(1-x)$ with $\gamma=0$ (first row), $\gamma=0.5$ (second row), $\gamma=1$ (third row),
       and different values  of $(\alpha,\beta)$.}\label{fig:figs2}
\end{figure}

\vspace{0.36cm}
One can expect that a sharp bounds for the Gegenbauer coefficients (\ref{eq:gegexpan}) and Chebyshev  coefficients (\ref{eq:chebexpan}) 
 from (\ref{eq:gegcoeffs}) and (\ref{eq:checoeffs}), respectively.

\begin{corollary}
The  Gegenbauer  and Chebyshev expansion coefficients for $f(x)$ are  satisfied
\begin{itemize}
\item  $f(x)=(1-x)^{\gamma}\ln^{\mu}(1-x)g(x)$ and $g\in C^{\infty}[-1,1]$:
\begin{equation}\label{Genend1}
  |a_{n}(\lambda)|=\left\{\begin{array}{ll}
  {\cal O}\left(\ln^{\mu}(n)n^{-2\lambda-2\gamma}\right),&\mbox{$\lambda+\gamma> -\frac{1}{2}$ and $\gamma\notin {\cal N}_0$}\\
  {\cal O}\left(\ln^{\mu-1}(n)n^{-2\lambda-2\gamma}\right),&\mbox{$\lambda+\gamma> -\frac{1}{2}$ and $\gamma\in {\cal N}_0$ },\end{array}\right.
  \end{equation}
\begin{equation}\label{Chebend1}
  |c_{n}|=\left\{\begin{array}{ll}
  {\cal O}\left(\ln^{\mu}(n)n^{-1-2\gamma}\right),&\mbox{ $\gamma> -\frac{1}{2}$ and $\gamma\notin {\cal N}_0$}\\
  {\cal O}\left(\ln^{\mu-1}(n)n^{-1-2\gamma}\right),&\mbox{$\gamma> -\frac{1}{2}$ and $\gamma\in {\cal N}_0$}.\end{array}\right.
\end{equation}

\item  $f(x)=(1+x)^{\delta}\ln^{\mu}(1+x)g(x)$  and $g\in C^{\infty}[-1,1]$:
\begin{equation}\label{Genend2}
  |a_{n}(\lambda)|=\left\{\begin{array}{ll}
  {\cal O}\left(\ln^{\mu}(n)n^{-2\lambda-2\delta}\right),&\mbox{$\lambda+\delta> -\frac{1}{2}$ and $\delta\notin {\cal N}_0$ }\\
  {\cal O}\left(\ln^{\mu-1}(n)n^{-2\lambda-2\delta}\right),&\mbox{$\lambda+\delta> -\frac{1}{2}$ and $\delta\in {\cal N}_0$},\end{array}\right.
  \end{equation}
  \begin{equation}\label{Chebend2}
    |c_{n}|=\left\{\begin{array}{ll}
  {\cal O}\left(\ln^{\mu}(n)n^{-1-2\delta}\right),&\mbox{$\delta> -\frac{1}{2}$ and $\delta\notin {\cal N}_0$ }\\
  {\cal O}\left(\ln^{\mu-1}(n)n^{-1-2\delta}\right),&\mbox{$\delta> -\frac{1}{2}$ and $\delta\in {\cal N}_0$}.\end{array}\right.
\end{equation}

\end{itemize}
\end{corollary}

{\sc Remark 2}
(i) Theorems 1-2 and Corollaries 1-2 show that  for functions of endpoint singularities, the decay of the coefficients  in a Legendre polynomial series has the same asymptotic order as the Chebyshev.

(ii) One may get faster convergence rate by increasing $(\alpha,\beta)$ for  endpoint singularities. In the case $\min\{\alpha,\beta\}>0$ or $\lambda>\frac{1}{2}$, the decay order of coefficients in a Jacobi or  Gegenbauer series is faster than
that in a Legendre or Chebyshev series.

Numerical results for these estimates on the two end points about the Chebyshev, Legendre and Jacobi expansion coefficients are illustrated in {\sc Figs}. \ref{fig:figs3} and  \ref{fig:figs4}, which indicates the optimal orders of the estimates. To get more clearly, in
the second row of {\sc Figs}. \ref{fig:figs3}, we consider $n=10:10:2000$.

\begin{figure}[hpbt]
      \centerline{\includegraphics[height=5cm,width=16cm]{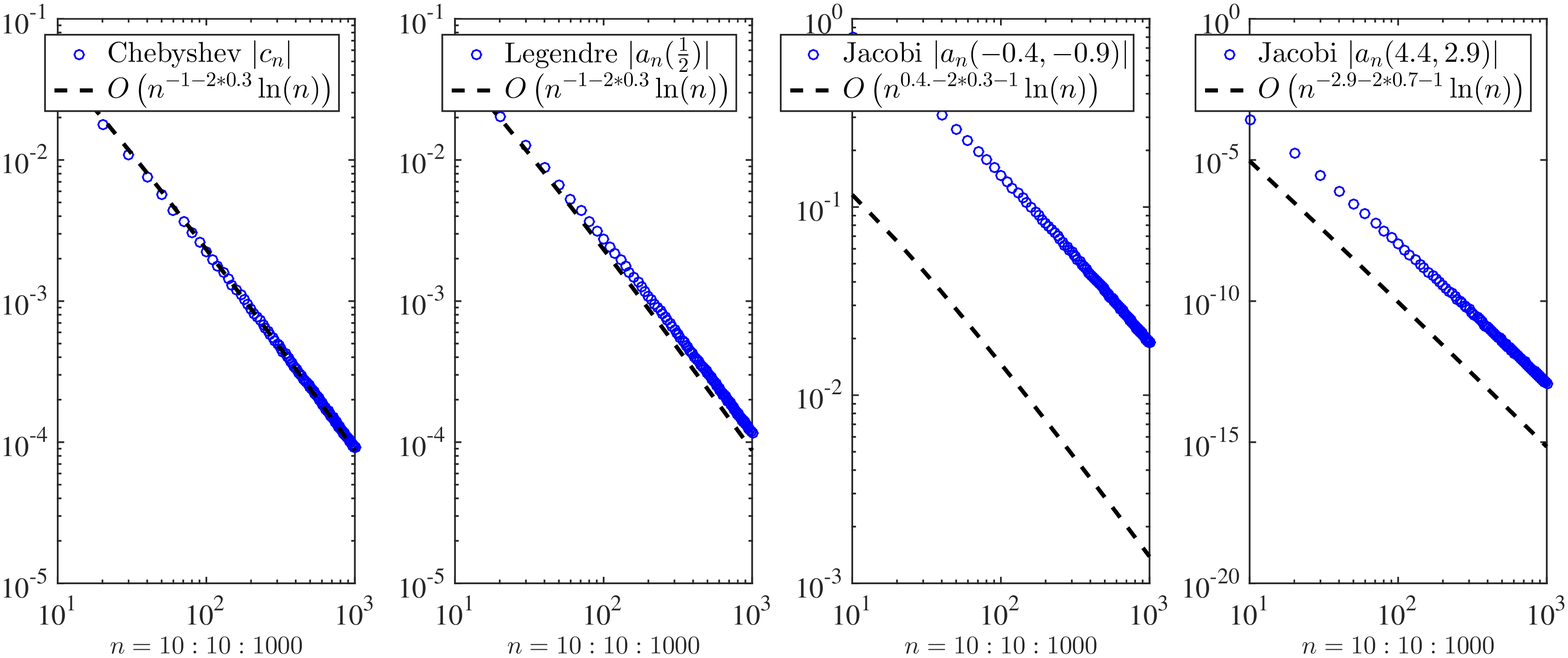}}
       \centerline{\includegraphics[height=5cm,width=16cm]{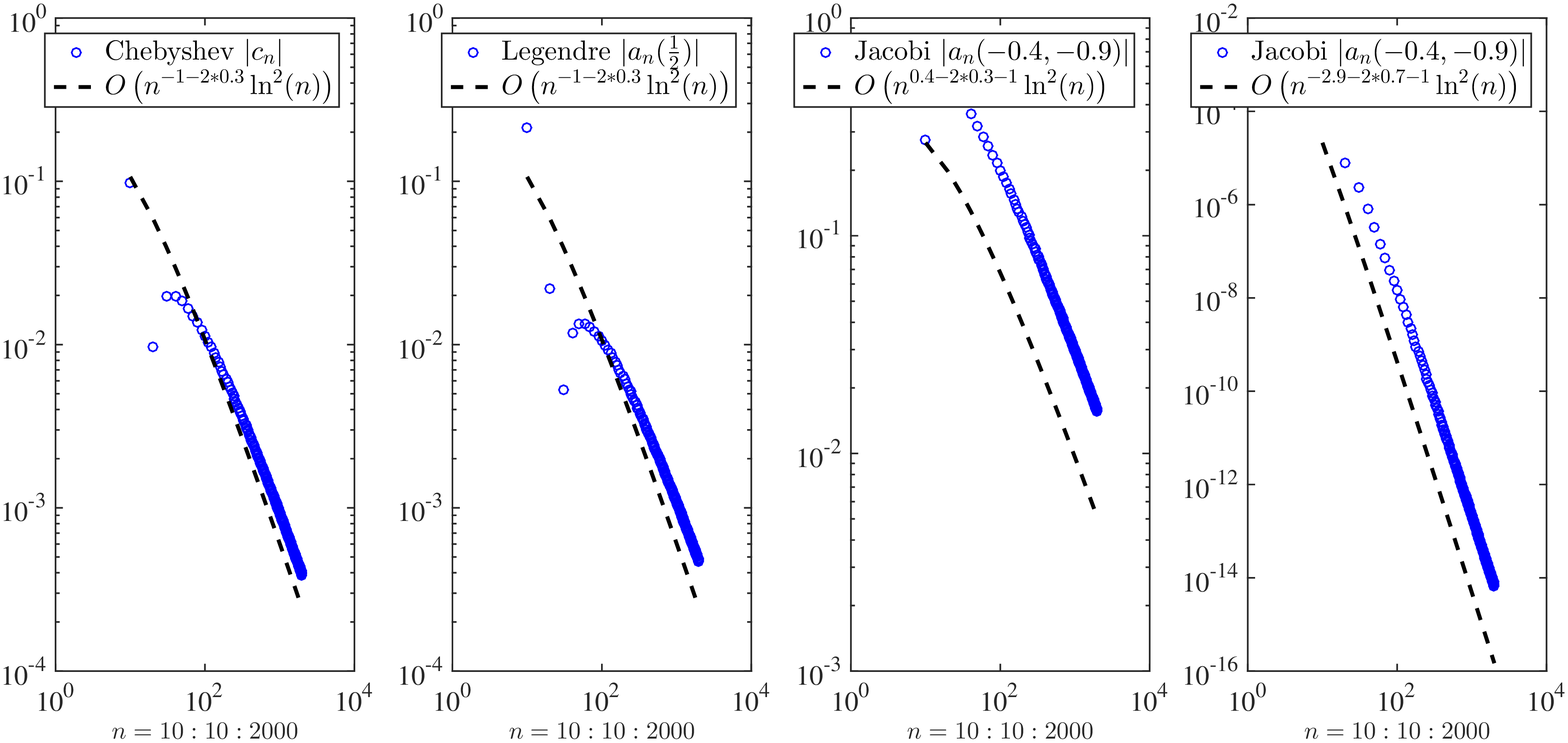}}
      \caption{The asymptotic decay of the Chebyshev, Legendre and Jacobi expansion coefficients  for  $f(x)=(1-x)^{0.3}(1+x)^{0.7}\ln^{\mu}(1-x^2)\sin{x}$ with different values of $(\alpha,\beta)$, respectively.: $\mu=1$ (first row) and $\mu=2$ (second row).}\label{fig:figs3}
\end{figure}

\begin{figure}[hpbt]
      \centerline{\includegraphics[height=5cm,width=16cm]{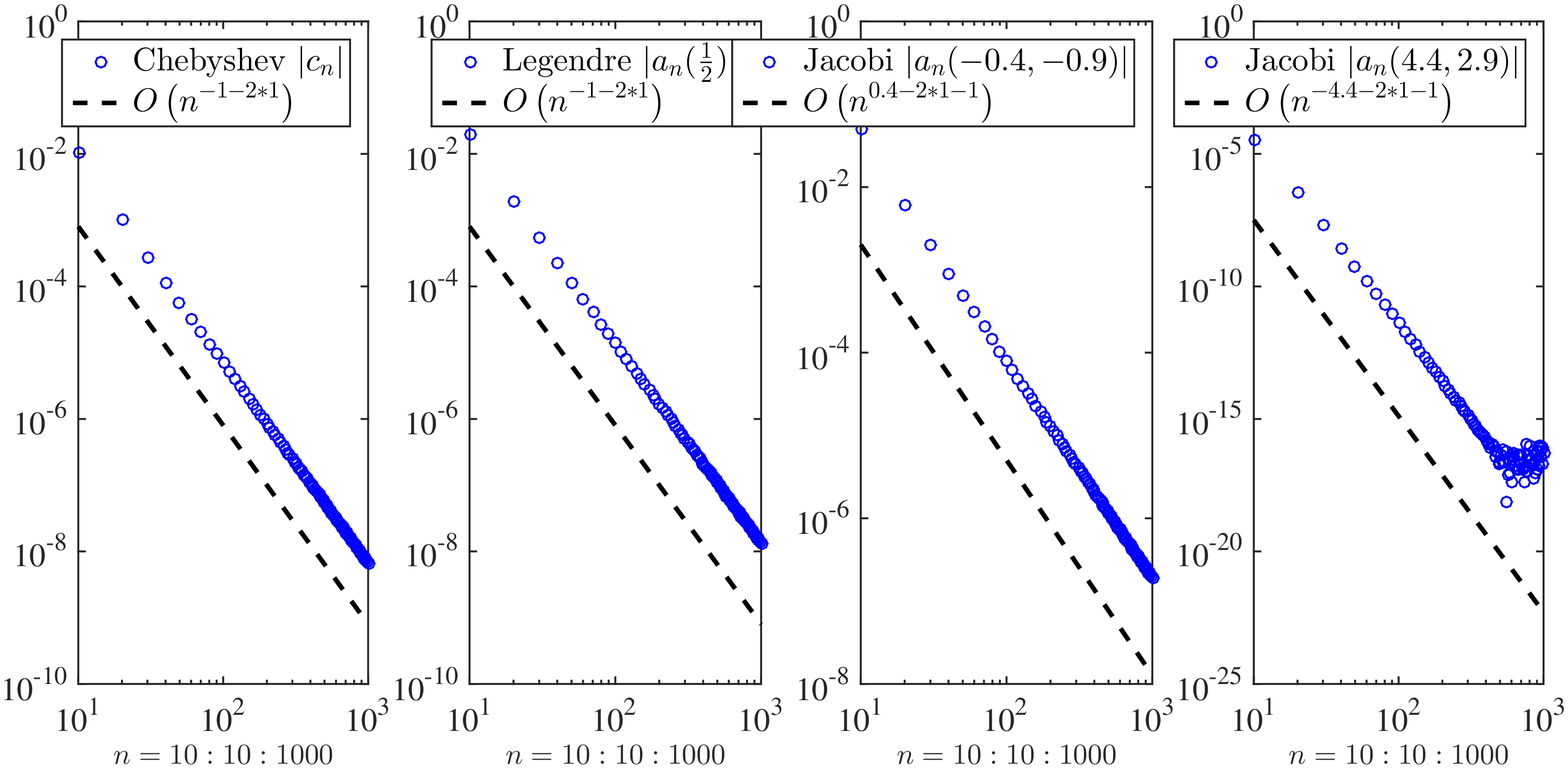}}
        \centerline{\includegraphics[height=5cm,width=16cm]{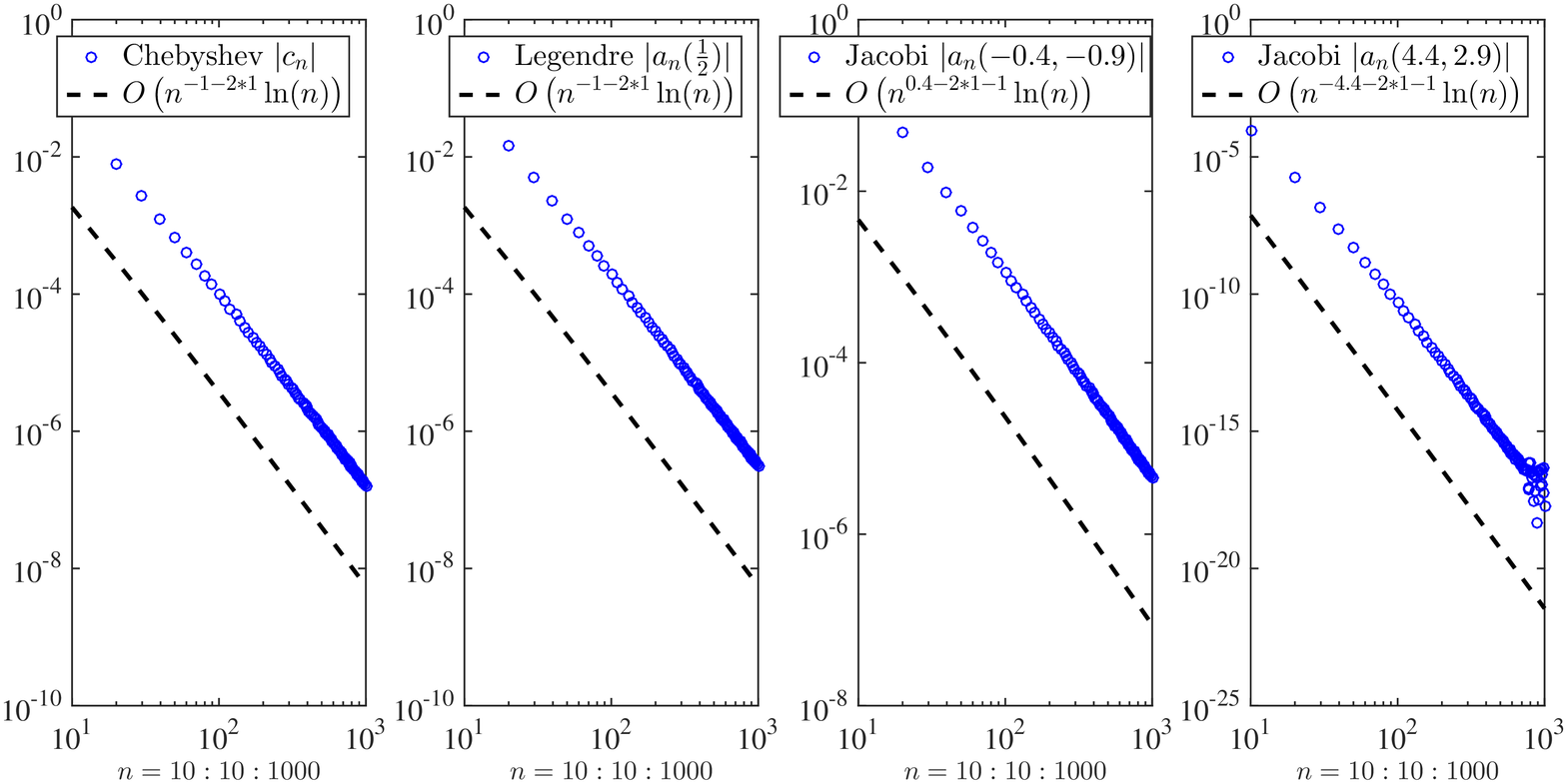}}
      \caption{The asymptotic decay of the Chebyshev, Legendre and Jacobi expansion coefficients  for  $f(x)=(1-x)^{1}(1+x)^{2}\ln^{\mu}(1-x^2)\sin{x}$ with different values of $(\alpha,\beta)$, respectively: $\mu=1$ (first row) and $\mu=2$ (second row).}\label{fig:figs4}
\end{figure}

\subsection{Functions with interior regularities}
Let us  consider
\begin{equation}\label{eq:intsingularfunc}
f(x)=|x-z_0|^s\ln^{\mu}|z-z_0|g(x),\quad z_{0}\in(-1,1)
\end{equation}
with $g\in C^{\infty}[-1,1]$, $s> 0$ a real number and  $\mu$ a positive integer.

\begin{theorem}
  Suppose that $f(x)$ is defined as (\ref{eq:intsingularfunc}), then the Jacobi coefficients (\ref{eq:jacexpcoeffs})  satisfy that
  \begin{equation}\label{jacasy}
    |a_{n}(\alpha,\beta)|={\cal O}\left(\ln^{\mu}(n)n^{-s-\frac{1}{2}}\right),\quad\mathrm{as}\ n\to\infty.
  \end{equation}
\end{theorem}

\begin{proof}
 Note that
  \begin{equation}\label{coejac}\quad\quad
    a_{n}(\alpha,\beta)=\frac{1}{\sigma_{n}^{\alpha,\beta}}\left[\int_{-1}^{z_0}+\int_{z_0}^{1}\right](1-x)^{\alpha}(1+x)^{\beta}P_{n}^{(\alpha,\beta)}(x)f(x)\mathrm{d}x.
  \end{equation}
  Without loss of generality, here we consider only the second integral in (\ref{coejac}). 
  
  Let $k_0=\lfloor s \rfloor$,  the greatest integer less than or equals to $s$. Since 
 $$\begin{array}{ll}
f^{(k_0)}(x)&=(x-z_0)^{s-k_0}\left[\ln^{\mu}(x-z_0)h_{\mu}(x)+\ln^{\mu-1}(1-x)h_{\mu-1}(x)\right.\\
&\quad \left.+\cdots+\ln(x-z_0)h_{1}(x)+h_0(x)\right]\\
&=:(x-z_0)^{s-k_0}\psi(x)\end{array}$$
with $h_j\in C^{\infty}[z_0,1],\, j=0,1,\ldots,\mu$.

 Applying Rodrigues' formula  follows
 $$
 I_1= {\displaystyle  \frac{\int_{-1}^{1}\!(1-x)^{\alpha+k_0}(1+x)^{\beta+k_0}P_{n-k_0}^{(\alpha+k_0,\beta+k_0)}(x)
  f^{(k_0)}(x)\,\mathrm{d}x}{(-1)^{k_0}2^{k_0}\sigma_{n}^{\alpha,\beta}n(n-1)\cdots(n-k_0+1)}}.
$$
 For simplicity, we only consider the first term  in $f^{(k_0)}(x)$ denoted by $\varphi_1(x)=(x-z_{0})^{s-k_0}\ln^{\mu}(x-z_0)\psi_1(x)$. Similar proof can be directly applied to estimate the other terms.  Without loss of generality,
 assume $f^{(k_0)}(x)=\varphi_1(x)=(x-z_{0})^{s-k_0}\ln^{\mu}(x-z_0)\psi_1(x)$.
 
In the case $s=k_0 $ is a positive integer:    Setting $x=\cos{\theta}$ and $\theta_0=\arccos{z_0}$, together with Lemma \ref{lem:lemma1}, analogously to the proofs of Theorem 1,  
 it indicates that  by Lemma 9 
$$\begin{array}{lll}
    I_1&=& {\displaystyle  \frac{\int_{0}^{\theta_0}\!(1-\cos{\theta})^{\alpha+k_0}(1+\cos{\theta})^{\beta+k_0}
    P_{n-k_0}^{(\alpha+k_0,\beta+k_0)}(\cos{\theta})\varphi_1(\cos\theta)\sin{\theta}\,\mathrm{d}\theta}{(-1)^{k_0}2^{k_0}\sigma_{n}^{\alpha,\beta}n(n-1)\cdots(n-k_0+1)}}\\
 &=& {\displaystyle \frac{\Gamma(n+\alpha+1)\int_{0}^{\theta_{0}}
    \left(\frac{\theta}{2}\right)^{\frac{1}{2}}\sin^{\alpha+k_0+\frac{1}{2}}\frac{\theta}{2}\cos^{\beta+k_0+\frac{1}{2}}\frac{\theta}{2}
    J_{\alpha+k_0}({\tilde N}\theta)\varphi_1(\cos\theta)\,\mathrm{d}\theta}{(-1)^{k_0}2^{-\alpha-\beta-1-k_0}(n-k_0)!{\tilde N}^{\alpha+k_0}\sigma_{n}^{\alpha,\beta}{n(n-1)\cdots(n-k_0+1)}}} \\
    & &\displaystyle +{\cal O}({\tilde N}^{-3/2-k_0+1})\\
    &=&{\displaystyle \frac{(-1)^{k_0}2^{\alpha+\beta+1+k_0}\Gamma(n+\alpha+1)}{n!{\tilde N}^{\alpha+k_0}\sigma_{n}^{\alpha,\beta}}
                           \int_{0}^{\theta_0}\theta^{\alpha+k_0+1}(\theta_0-\theta)^{s-k_0}\left\{\ln^{\mu}(\theta_0-\theta)\hat\varphi_{\mu}(\theta)\right.}\\
                           &&{\displaystyle \left.+\cdots+\ln(\theta_0-\theta)\hat\varphi_{1}(\theta)+\hat\varphi_0(\theta)\right\}
                           J_{\alpha+k_0}({\tilde N}\theta)\,\mathrm{d}\theta+O({\tilde N}^{-\frac{1}{2}-k_0})}\\
                           &=&{\displaystyle  {\cal O}\left(\ln^{\mu}(n)n^{-s-\frac{1}{2}}\right) +\cdots+{\cal O}\left(\ln(n)n^{-s-\frac{1}{2}}\right) +{\cal O}\left(n^{-s-\frac{1}{2}}\right)     +O(n^{-\frac{1}{2}-k_0})             }\\
   &=&{\displaystyle  {\cal O}\left(\ln^{\mu}(n)n^{-s-\frac{1}{2}}\right)} \end{array}
  $$
due to $\alpha+k_0+2>\frac{3}{2}$, where ${\tilde N}=n+(\alpha+\beta+1)/2$ and $\hat\varphi_j(\theta) \in C^{\infty}[0,\theta_0]$ for $j=0,1,\ldots,\mu$. 

In the case that $s>0$ is not an integer: By integrating by parts once again, it follows
 $$\begin{array}{lll}
    I_1&=& {\displaystyle  \frac{\int_{0}^{\theta_0}\!(1-\cos{\theta})^{\alpha+k_0+1}(1+\cos{\theta})^{\beta+k_0+1}
    P_{n-k_0-1}^{(\alpha+k_0+1,\beta+k_0+1)}(\cos{\theta})\varphi_1^{\prime}(\cos\theta)\sin{\theta}\,\mathrm{d}\theta}{(-1)^{k_0+1}2^{k_0+1}\sigma_{n}^{\alpha,\beta}n(n-1)\cdots(n-k_0)}}\\
 &=& {\displaystyle \frac{\Gamma(n+\alpha+1)\int_{0}^{\theta_{0}}
    \left(\frac{\theta}{2}\right)^{\frac{1}{2}}\sin^{\alpha+k_0+\frac{3}{2}}\frac{\theta}{2}\cos^{\beta+k_0+\frac{3}{2}}\frac{\theta}{2}
    J_{\alpha+k_0+1}({\tilde N}\theta)\varphi_1^{\prime}(\cos\theta)\,\mathrm{d}\theta}{(-1)^{k_0+1}2^{-\alpha-\beta-2-k_0}(n-k_0-1)!{\tilde N}^{\alpha+k_0}\sigma_{n}^{\alpha,\beta}{n(n-1)\cdots(n-k_0)}}} \\
    & &\displaystyle +{\cal O}({\tilde N}^{-3/2-k_0})\\
    &=&{\displaystyle \frac{(-1)^{k_0+1}2^{\alpha+\beta+2+k_0}\Gamma(n+\alpha+1)}{n!{\tilde N}^{\alpha+k_0+1}\sigma_{n}^{\alpha,\beta}}
                           \int_{0}^{\theta_0}\theta^{\alpha+k_0+2}(\theta_0-\theta)^{s-k_0-1}\left\{\ln^{\mu}(\theta_0-\theta)\hat\varphi_{\mu}(\theta)\right.}\\
                           &&{\displaystyle \left.+\cdots+\ln(\theta_0-\theta)\hat\varphi_{1}(\theta)+\hat\varphi_0(\theta)\right\}
                           J_{\alpha+k_0+1}({\tilde N}\theta)\,\mathrm{d}\theta+O({\tilde N}^{-\frac{3}{2}-k_0})}\\
                           &=&{\displaystyle  {\cal O}\left(\ln^{\mu}(n)n^{-s-\frac{1}{2}}\right) +\cdots+{\cal O}\left(\ln(n)n^{-s-\frac{1}{2}}\right) +{\cal O}\left(n^{-s-\frac{1}{2}}\right)     +O(n^{-\frac{3}{2}-k_0})             }\\
   &=&{\displaystyle  {\cal O}\left(\ln^{\mu}(n)n^{-s-\frac{1}{2}}\right)}, \end{array}
  $$
 by Lemma \ref{lemma9}  and  $\alpha+2k_0+3>k_0+\frac{3}{2}>s+\frac{1}{2}$.

    Similar results can be derived for the first integral in the right hand side of (\ref{coejac}). These lead to the desired result (\ref{jacasy}).

\end{proof}

{\sc Remark 3}.
  In the case $-1<s\le 0$, in analogy to the above proof without integrating by parts, from Lemma 9 we get that
   \begin{equation}\label{jacasy2}
    |a_{n}(\alpha,\beta)|={\cal O}\left(\max\left\{\ln^{\mu}(n)n^{-s-\frac{1}{2}},n^{-\min\{1+\alpha,1+\beta\}}\right\}\right),\quad\mathrm{as}\ n\to\infty.
  \end{equation}
  

\begin{corollary}
 Suppose that $f(x)$ is defined as (\ref{eq:intsingularfunc}), then for the  Gegenbauer  and Chebyshev expansion coefficients, the following holds
\begin{equation}\label{Genint1}
  |a_{n}(\lambda)|=
  {\cal O}\left(\ln^{\mu}(n)n^{-\lambda-s}\right),
 \end{equation}
\begin{equation}\label{Chebint1}
    |c_{n}|=
  {\cal O}\left(\ln^{\mu}(n)n^{-1-s}\right).
  \end{equation}
\end{corollary}

{\sc Remark 4}.   (i) For functions of interrior singularities,  the decay order of the coefficients  in a Jacobi polynomial series is independent of $(\alpha,\beta)$ if $s>0$. 

(ii) The decay of the coefficients  in a Chebyshev polynomial series has the highest asymptotic order. {\sc Fig.} \ref{fig:figs5} illustrates this phenomenon exactly.  
However, in Section 4, we will see that all these spectral expansions have the same convergence order.

\begin{figure}[hpbt]
      \centerline{\includegraphics[height=5cm,width=15cm]{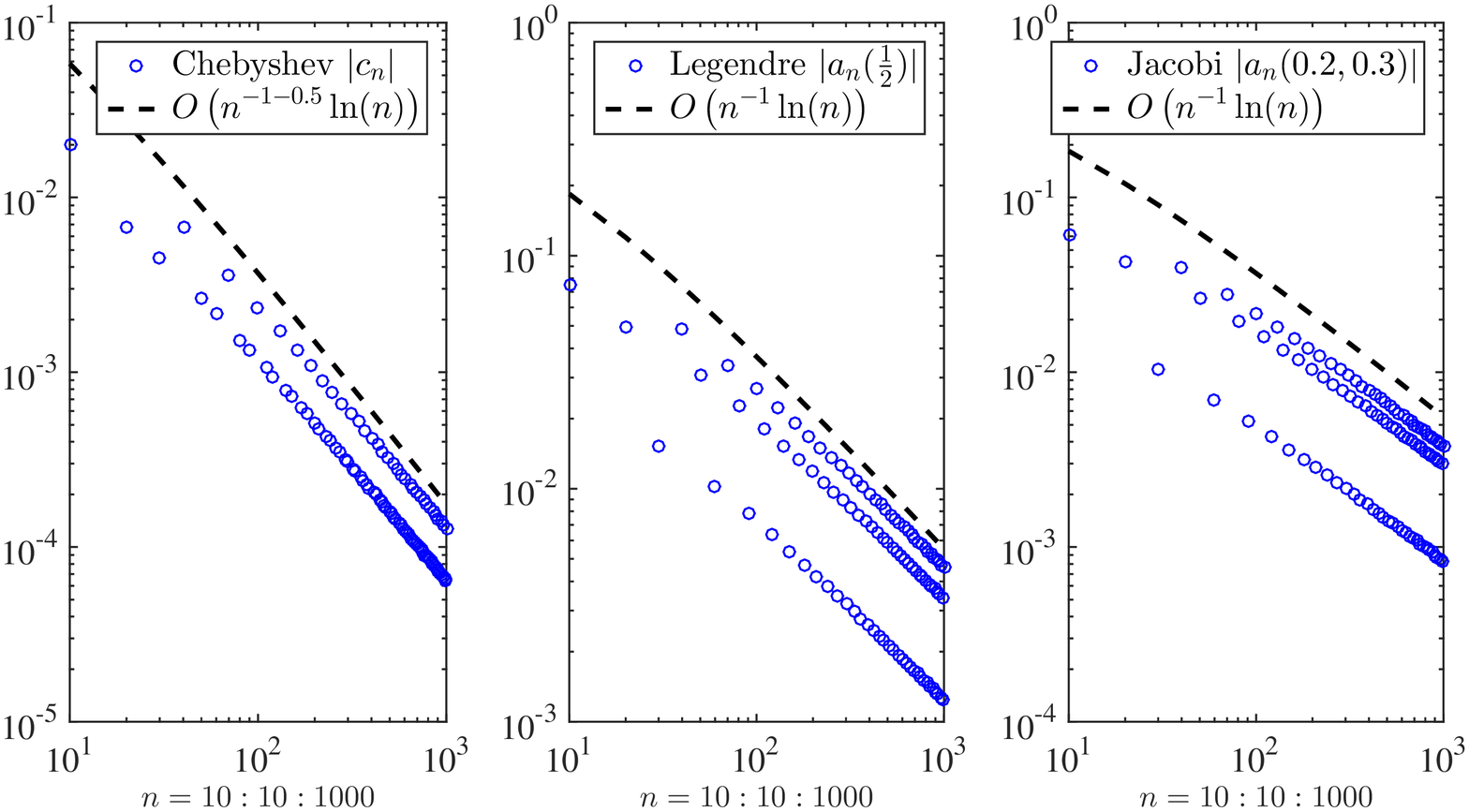}}
        \centerline{\includegraphics[height=5cm,width=15cm]{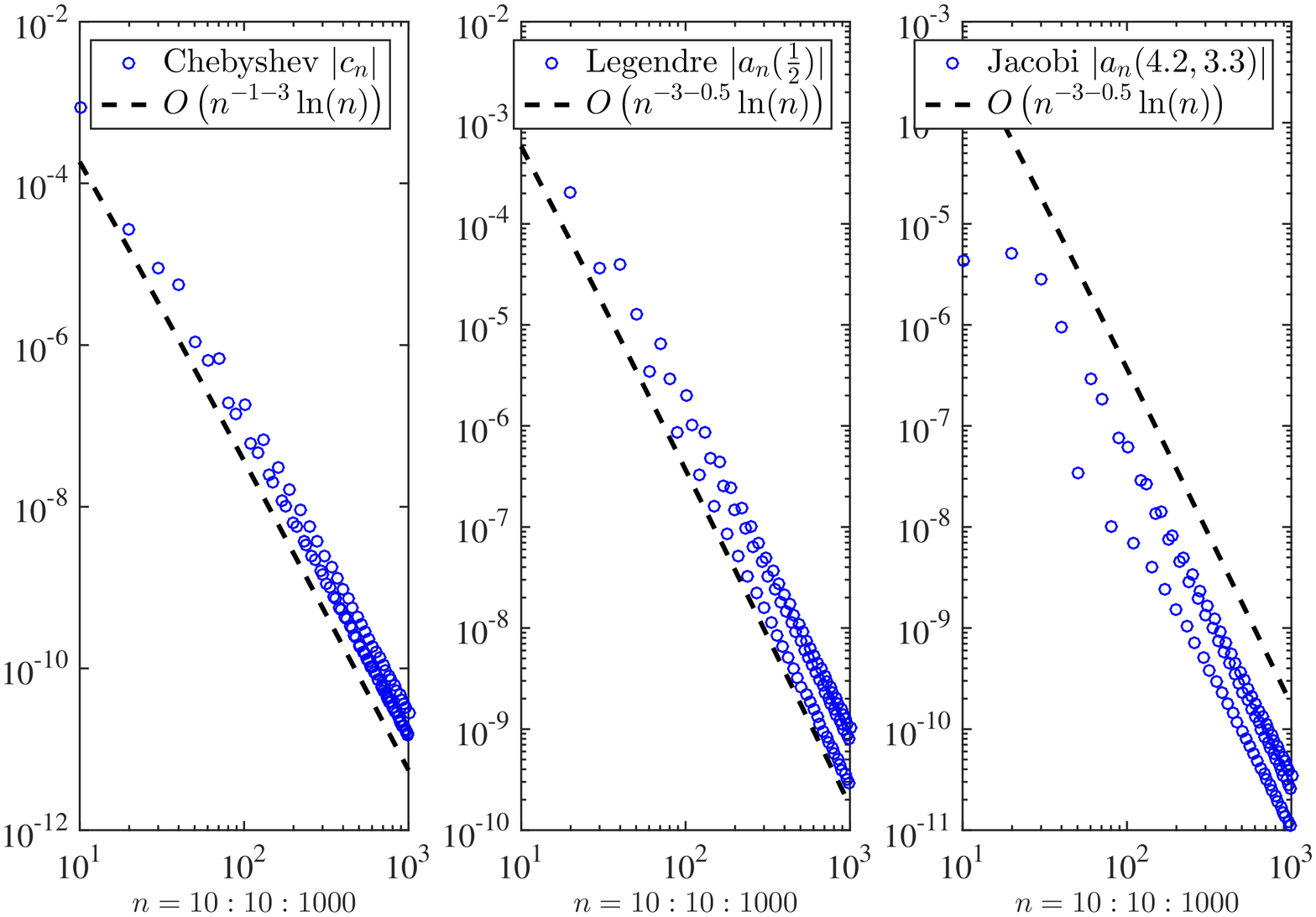}}
      \caption{The asymptotic decay of the Chebyshev, Legendre Jacobi coefficients  for  $f(x)=|x-0.5|^{s}\ln|x-0.5|\cos{x}$ with $s=0.5$ and $s=3$, respectively: $n=10:10:1000$.}\label{fig:figs5}
\end{figure}

{\sc Remark 5}. From the proofs, we see that  all the above thereoms  and corollaries still hold for $\mu=0$ \cite{XiangLiu} except for $s$ even. In the case $\mu=0$ and $s$ even, $a_n(\alpha,\beta)$,  $a_{n}(\lambda)$ and $c_n$ are exponentially decayed \cite{Bernstein,Trefethen,Wang1,Wang2,Xiang2012,XWZ,ZWX}.

\section{The convergence rates on the spectral orthogonal projections}

For $(x)=(1-x)^{\gamma}\ln^{\mu}(1-x)g(x)$ or $f(x)=(1+x)^{\delta}\ln^{\mu}(1+x)g(x)$, it is easy to verify that $f\in L_{\omega}^{2}[-1,1]$ if $\min\{\alpha+2\gamma,\beta+2 \delta\}>-1$. For $f(x)=|x-z_0|^s\ln^{\mu}|z-z_0|g(x)$, 
 $f\in L_{\omega}^{2}[-1,1]$ if $s>-\frac{1}{2}$.
 
\begin{theorem}\label{thm41}
Suppose that $f$ is of algebraic and logarithmatic regularity at an endpoint or  interior point, then for the Jacobi expansion, it follows that
\begin{equation}
  \|f-\mathcal{P}_{N}^{f,Ja}\|_{L_{\omega}^{2}[-1,1]}=\left\{\begin{array}{ll}
   {\cal O}(N^{-\alpha-2\gamma-1}\ln^{\mu}(N)),& \mbox{$f(x)=(1-x)^{\gamma}\ln^{\mu}(1-x)g(x)$}\\
   {\cal O}(N^{-s-1/2}\ln^{\mu}(N)),&\mbox{$f(x)=|x-z_0|^s\ln^{\mu}|z-z_0|g(x)$}\\
   {\cal O}(N^{-\beta-2\delta-1}\ln^{\mu}(N)),& \mbox{$f(x)=(1+x)^{\delta}\ln^{\mu}(1+x)g(x)$}\end{array}\right.
\end{equation}
where $g\in C^{\infty}[-1,1]$,  $\mu$ is a nonnegative integer, $z_{0}\in(-1,1)$, $\min\{\alpha+\gamma,\beta+ \delta,\alpha+2\gamma,\beta+ 2\delta\}>-1$ for the boundary singularities, 
and   $s> -\frac{1}{2}$ and $\min\{\alpha,\beta\}\ge  -\frac{1}{2}$ for  the interior singularity.  In particular, if $\gamma,\, \delta$ are integers and $\mu\ge 1$, then 
\begin{equation}
  \|f-\mathcal{P}_{N}^{f,Ja}\|_{L_{\omega}^{2}[-1,1]}=\left\{\begin{array}{ll}
    {\cal O}(N^{-\alpha-2\gamma-1}\ln^{\mu-1}(N)),& \mbox{$f(x)=(1-x)^{\gamma}\ln^{\mu}(1-x)g(x)$}\\
{\cal O}(N^{-\beta-2\delta-1}\ln^{\mu-1}(N)),& \mbox{$f(x)=(1+x)^{\delta}\ln^{\mu}(1+x)g(x)$}.\end{array}\right.
\end{equation}
\end{theorem}

While the optimal convergence rate on the orthogonal projection can be obtained from 
the sharp bounds for the Gegenbauer coefficients  and
$$
\|f-\mathcal{P}_{N}^{f,Ge}\|_{L_{\omega}^{2}[-1,1]}=\sqrt{\sum_{n=N+1}^{\infty}a_{n}^{2}(\lambda)\ \hbar_{n}}
$$
together with 
$$
 \hbar_{n}=\frac{2^{1-2\lambda}\pi}{\Gamma^{2}(\lambda)}\frac{\Gamma(n+2\lambda)}{n!(n+\lambda)}={\cal O}(n^{2\lambda-2})
$$
(see \cite[p. 79]{Hesthaven} and (\cite[(7.33.1) p. 171]{Szego}).

\begin{corollary}
Suppose that $f$ is of  algebraic and logorithmatic regularity at an endpoint or  interior point, then for the Gegenbauer and Chebyshev  expansions, it follows that

\begin{equation}
\quad\quad  \|f-\mathcal{P}_{N}^{f,Ge}\|_{L_{\omega}^{2}[-1,1]}=\left\{\begin{array}{ll}
   {\cal O}\left(N^{-\lambda-2\gamma-1/2}\ln^{\mu}(N)\right),& \mbox{$f(x)=(1-x)^{\gamma}\ln^{\mu}(1-x)g(x)$}\\
   {\cal O}(N^{-s-1/2}\ln^{\mu}(N)),&\mbox{$f(x)=|x-z_0|^s\ln^{\mu}|z-z_0|g(x)$}\\
    {\cal O}\left(N^{-\lambda-2\delta-1/2}\ln^{\mu}(N)\right),& \mbox{$f(x)=(1+x)^{\delta}\ln^{\mu}(1+x)g(x)$},\end{array}\right.
\end{equation}
\begin{equation}
\quad\quad\|f-\mathcal{P}_{N}^{f,Ch}\|_{L_{\omega}^{2}[-1,1]}=\left\{\begin{array}{ll}
   {\cal O}\left(N^{-1/2-2\gamma}\ln^{\mu}(N)\right),& \mbox{$f(x)=(1-x)^{\gamma}\ln^{\mu}(1-x)g(x)$}\\
   {\cal O}(N^{-s-1/2}\ln^{\mu}(N)),&\mbox{$f(x)=|x-z_0|^s\ln^{\mu}|z-z_0|g(x)$}\\
    {\cal O}\left(N^{-1/2-2\delta}\ln^{\mu}(N)\right),& \mbox{$f(x)=(1+x)^{\delta}\ln^{\mu}(1+x)g(x)$}\end{array}\right.
\end{equation}
where $g\in C^{\infty}[-1,1]$, $z_{0}\in(-1,1)$, $\min\{\lambda+\gamma,\lambda+2\gamma,\lambda+ \delta,,\lambda+2 \delta\}>-\frac{1}{2}$,  $s>-\frac{1}{2}$ with $\lambda\ge 0$.  In particular, if $\gamma,\, \delta$ are integers and $\mu\ge 1$, then 
\begin{equation}
\quad\quad  \|f-\mathcal{P}_{N}^{f,Ge}\|_{L_{\omega}^{2}[-1,1]}=\left\{\begin{array}{ll}
    {\cal O}(N^{-\lambda-2\gamma-1/2}\ln^{\mu-1}(N)),& \mbox{$f(x)=(1-x)^{\gamma}\ln^{\mu}(1-x)g(x)$}\\
 {\cal O}\left(N^{-\lambda-2\delta-1/2}\ln^{\mu-1}(N)\right),& \mbox{$f(x)=(1+x)^{\delta}\ln^{\mu}(1+x)g(x)$},\end{array}\right.
\end{equation}
\begin{equation}
\quad\quad \|f-\mathcal{P}_{N}^{f,Ch}\|_{L_{\omega}^{2}[-1,1]}=\left\{\begin{array}{ll}
   {\cal O}\left(N^{-1/2-2\gamma}\ln^{\mu-1}(N)\right),& \mbox{$f(x)=(1-x)^{\gamma}\ln^{\mu}(1-x)g(x)$}\\
    {\cal O}\left(N^{-1/2-2\delta}\ln^{\mu-1}(N)\right),& \mbox{$f(x)=(1+x)^{\delta}\ln^{\mu}(1+x)g(x)$}.\end{array}\right.
\end{equation}
\end{corollary}

{\sc Remark 6}.
From Theorem 4 and Corollary 4, we see that for functions with limited regularities at the endpoints $f(x)=(1-x)^\gamma(1+x)^\delta\ln^{\mu}(1-x^2) g(x)$, the Jacobi or Gegenbauer expansion can achieve 
faster convergence order than Chebyshev expansion if $\min\{\alpha,\beta\}>-\frac{1}{2}$ or $\lambda>0$,  which is illustrated by {\sc Fig.} \ref{fig:figs6}.
\begin{figure}[hpbt]
\centerline{\includegraphics[height=4cm,width=15cm]{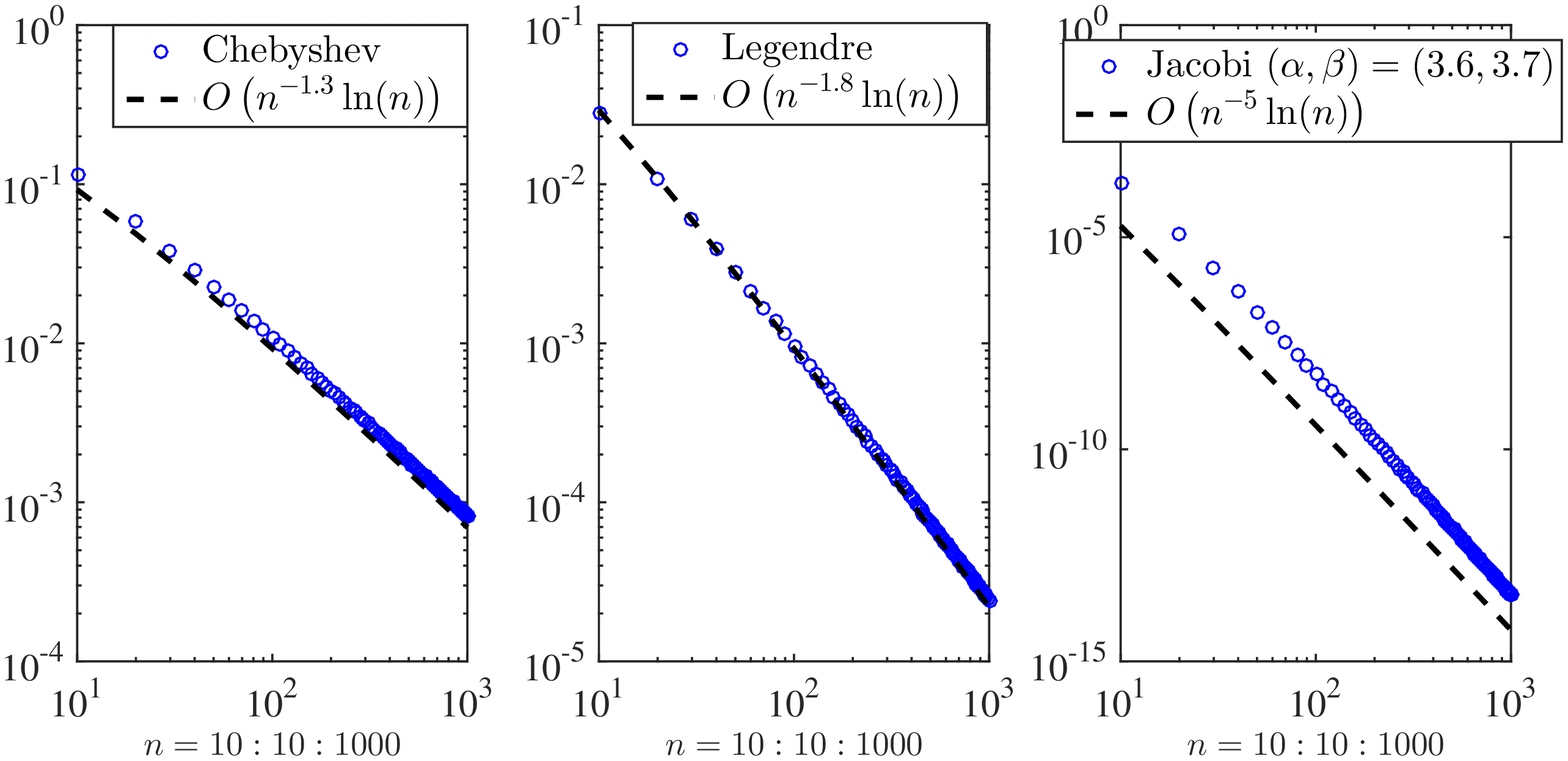}}
         \centerline{\includegraphics[height=4cm,width=15cm]{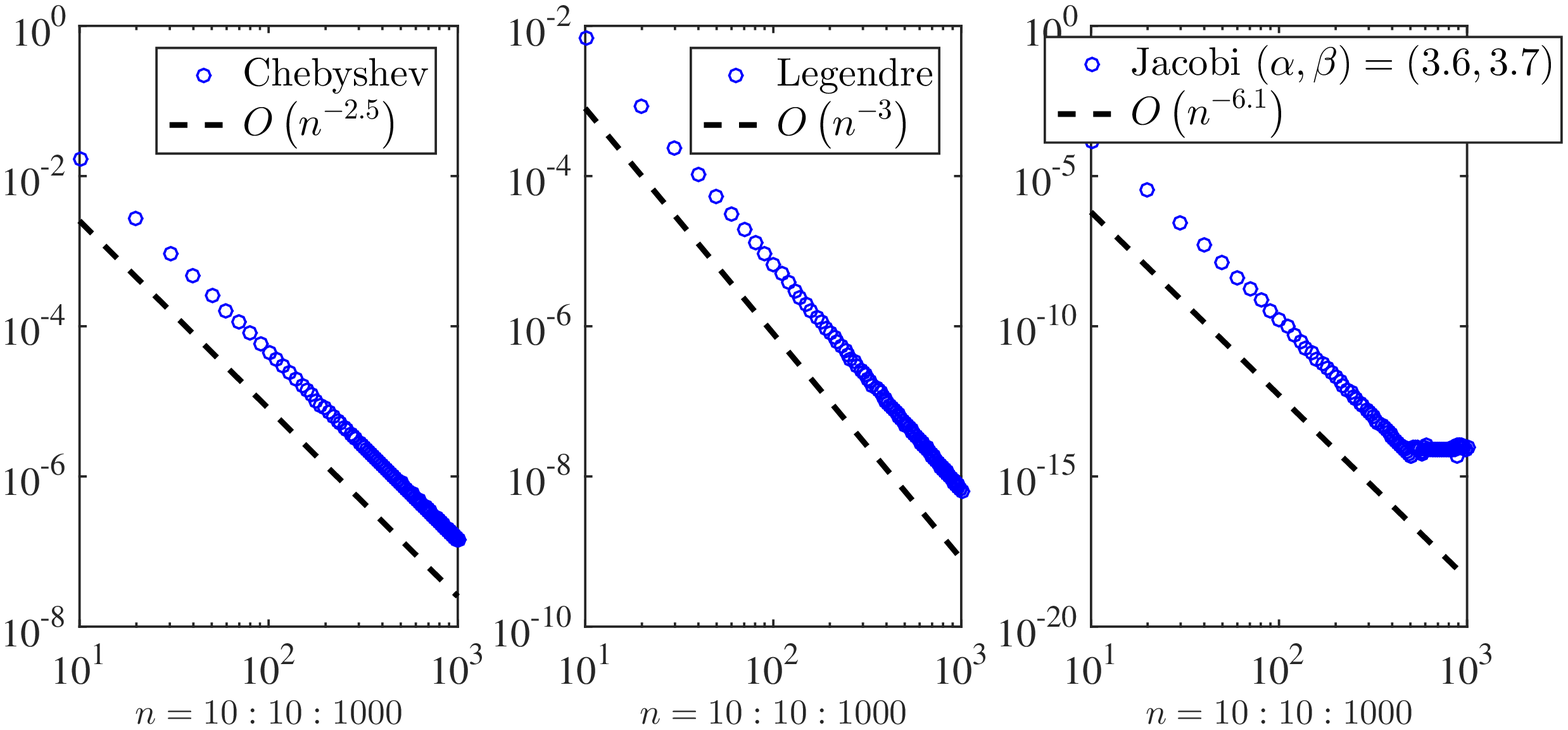}}
\caption{The weighted norm errors of the truncated Chebyshev, Legendre and Jacobi ($\alpha=3.6$, $\beta=3.7$)  expansions $\|f-\mathcal{P}_{n}^{f}\|_{L_{\omega}^{2}[-1,1]}$ for $f(x)=(1-x)^{\gamma}(1+x)^{\delta}\ln(1-x^2)$: $\gamma=0.6$ and $\delta=0.4$ (first row), and  $\gamma=1$ and $\delta=2$ (second row), respectively.}\label{fig:figs6}
\end{figure}

{\sc Remark 7}. It is obvious from Theorem 4 and Corollary 4 that for the functions with  interior  singularity, all these spectral expansions have the same convergence order in the case $s>-\frac{1}{2}$ with 
$\min\{\alpha,\beta\}\ge  -\frac{1}{2}$ or $\lambda\ge 0$. All these estimates are attainable.

\begin{figure}[hpbt]
\centerline{\includegraphics[height=4cm,width=15cm]{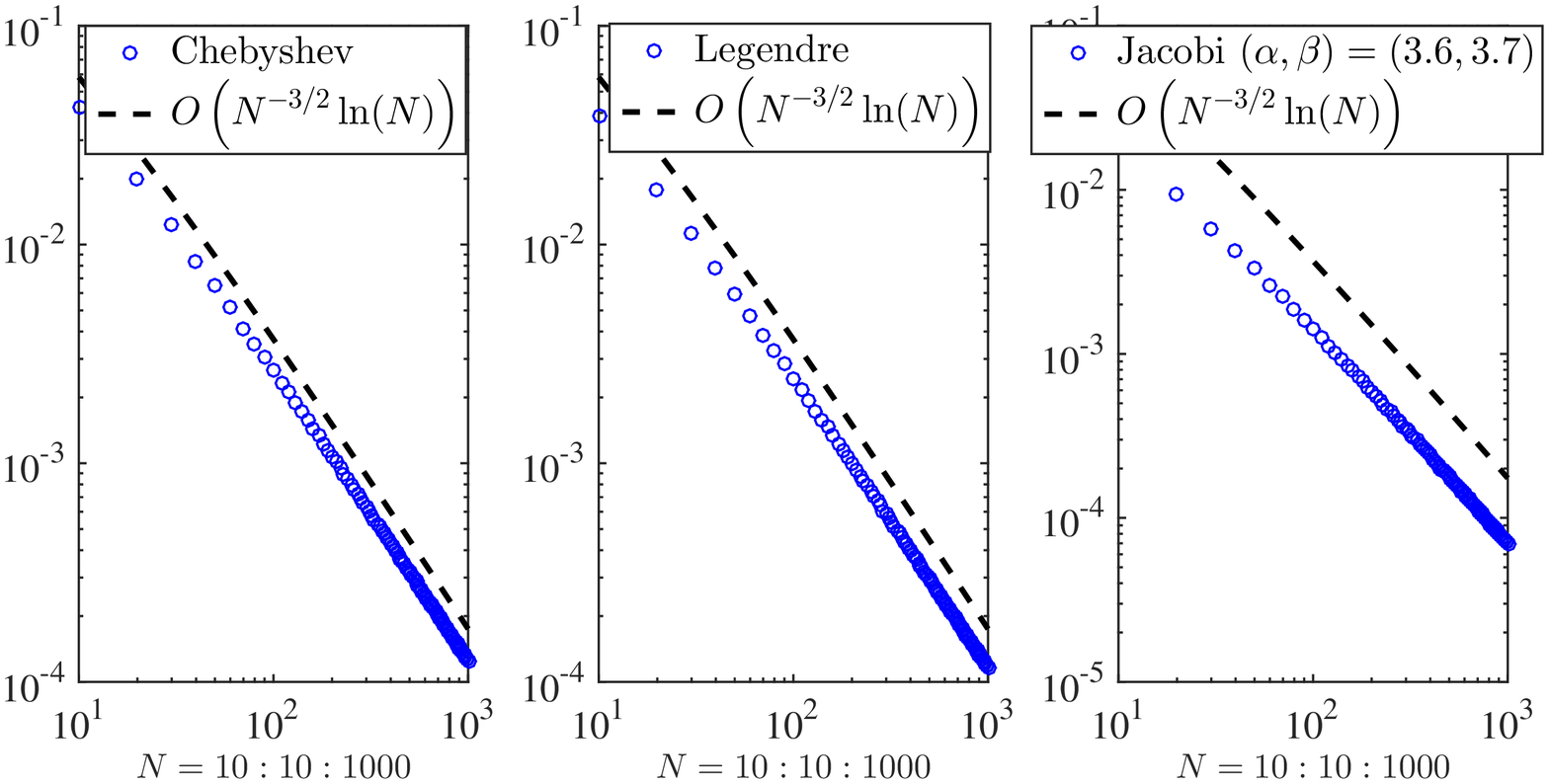}}
         \centerline{\includegraphics[height=4cm,width=15cm]{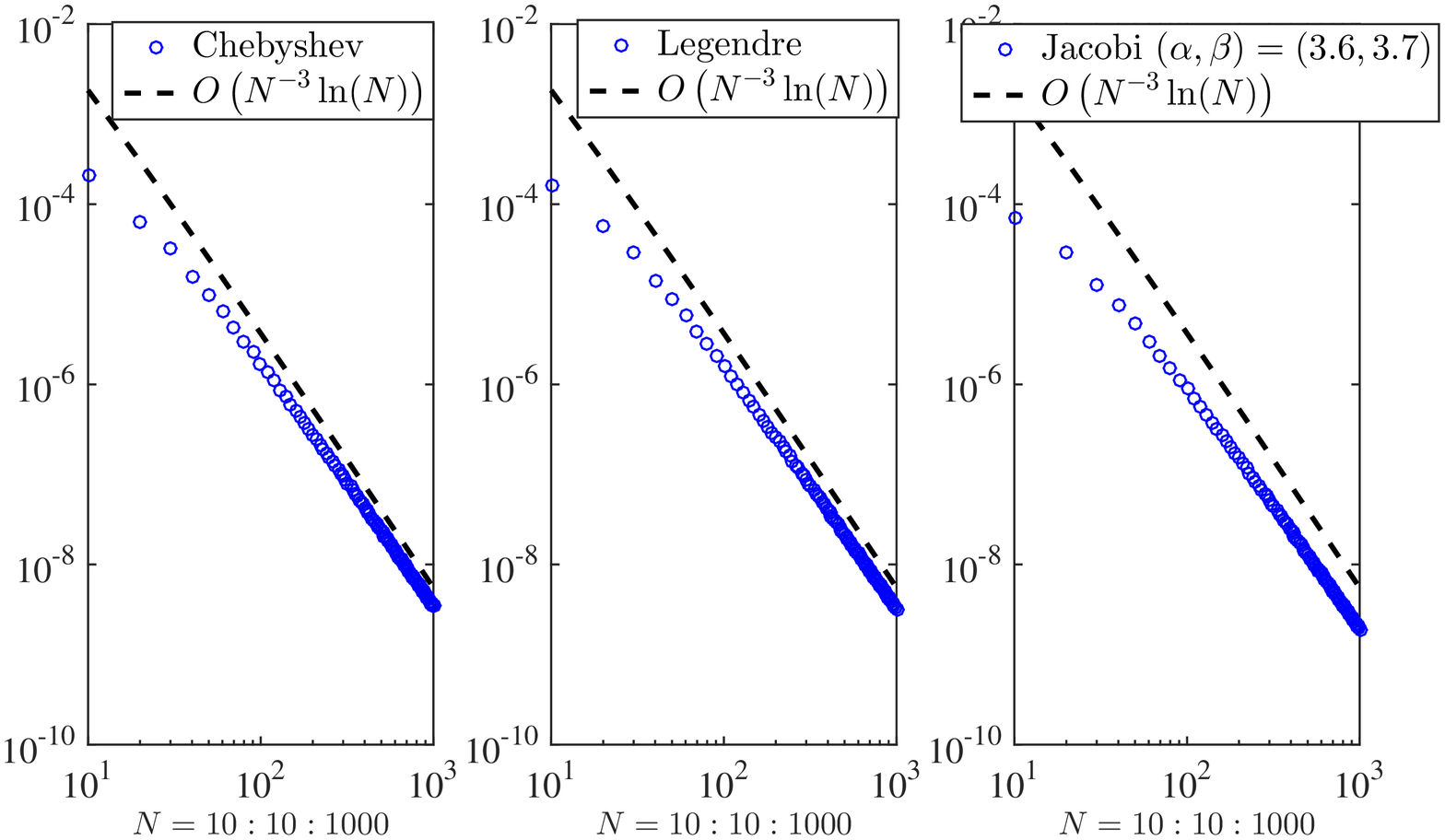}}
\caption{The weighted norm errors of the truncated Chebyshev, Legendre and Jacobi ($\alpha=3.6$, $\beta=3.7$) expansions $\|f-\mathcal{P}_{n}^{f}\|_{L_{\omega}^{2}[-1,1]}$ for $f(x)=|x-\frac{1}{2}|^s\ln|x-\frac{1}{2}|$: $s=1$ (first row) and $s=2.5$ (second row), respectively.}\label{fig:figs7}
\end{figure}

\vspace{0.6cm}
In particular, for  the non-uniformly Jacobi-weighted Sobolev space $H^{m,\alpha,\beta}(\Omega)$ with integer $m\ge   0$, $\alpha>-1$, $\beta >-1$, $\Omega=[-1,1]$ and weighted norm
\begin{equation}
\|u\|_{H^{m,\alpha,\beta}(\Omega)}=\left\{\sum_{q=0}^m\int_{-1}^1(1-x)^{\alpha+q}(1+x)^{\beta+q}[u^{(q)}(x)]^2dx\right\}^{\frac{1}{2}},
\end{equation}
define 
\begin{equation}\displaystyle
f^{(q)}(x)=\sum_{n=0}^{\infty}a^{(q)}_n(\alpha+q,\beta+q)P_n^{(\alpha+q,\beta+q)}(x),\quad q=0,1,\ldots,m
\end{equation}
with
\begin{equation}\displaystyle
a^{(q)}_n(\alpha+q,\beta+q)=\frac{1}{\sigma_n^{\alpha+q,\beta+q}}\int_{-1}^1(1-x)^{\alpha+q}(1+x)^{\beta+q}f^{(q)}(x)P_n^{(\alpha+q,\beta+q)}(x)dx.
\end{equation}
From (\ref{JacInt}), we see that
\begin{equation}\label{JacInt2}
\begin{array}{lll}
&&a^{(q)}_n(\alpha+q,\beta+q)\\
&=&\frac{1}{\sigma_n^{\alpha+q,\beta+q}}\int_{-1}^1(1-x)^{\alpha+q}(1+x)^{\beta+q}f^{(q)}(x)P_n^{(\alpha+q,\beta+q)}(x)dx\\
 &=&\frac{(-1)^{q}2^{q}\sigma_{n+q}^{\alpha,\beta}}{\sigma_n^{\alpha+q,\beta+q}}(n+q)(n+q-1)\cdots(n+1)\int_{-1}^{1}\!(1-x)^{\alpha}(1+x)^{\beta}f(x)P_{n+q}^{(\alpha,\beta)}(x)\,\mathrm{d}x\\
 &=& \frac{\sigma_{n+q}^{\alpha,\beta}}{\sigma_n^{\alpha+q,\beta+q}}(-1)^{q}2^{q}(n+q)(n+q-1)\cdots(n+1) a_{n+q}(\alpha,\beta)\\
&=&a_{n+q}(\alpha,\beta){\cal O}(n^q),
\end{array}
\end{equation}
which together with Theorems 1-3 yields the following convergence rate.

Notice that for $(x)=(1-x)^{\gamma}\ln^{\mu}(1-x)g(x)$ or $f(x)=(1+x)^{\delta}\ln^{\mu}(1+x)g(x)$, it is easy to verify that $f\in H^{m,\alpha,\beta}(\Omega)$ if $\min\{\alpha+2\gamma-m,\beta+2 \delta-m\}>-1$. For $f(x)=|x-z_0|^s\ln^{\mu}|z-z_0|g(x)$, 
 $f\in H^{m,\alpha,\beta}(\Omega)$ if $s>m-\frac{1}{2}$.

\begin{theorem}
For the Jacobi expansion, (\ref{JNerror1}) and (\ref{JNerror2}) are satisfied.
\end{theorem}

\begin{proof}
Note that
$$
 \|f-\mathcal{P}_{N}^{f,Ja}\|^2_{H^{m,\alpha,\beta}(\Omega)}=\sum_{q=0}^m\sum_{n=N+1}^{\infty}[a_n^{(q)}(\alpha+q,\beta+q)]^2\sigma_{n}^{\alpha+q,\beta+q},
 $$
 which directly leads to the desired result by Theorems 1-3 together with (\ref{JacInt2}).
\end{proof}

Similar results can be obtained for Gegenbauer and Chebyshev projections from $${\rm span}\left\{T_n(x)\right\}_{n=0}^N={\rm span}\left\{P^{(-\frac{1}{2},-\frac{1}{2})}_n(x)\right\}_{n=0}^N$$ and
$${\rm span}\left\{C_n^{(\lambda)}(x)\right\}_{n=0}^N={\rm span}\left\{P^{(\lambda-\frac{1}{2},\lambda-\frac{1}{2})}_n(x)\right\}_{n=0}^N.$$
 
 \begin{corollary}
For the Gegenbauer and Chebyshev  expansions, it follows for $\mu\ge 0$ that

\begin{equation}\label{eq:asytranjacexpanweights}
\quad \,\,\,\,\, \|f-\mathcal{P}_{N}^{f,Ge}\|_{H^{m,\alpha,\beta}(\Omega)}=\left\{\begin{array}{ll}
   {\cal O}\left(N^{m-\lambda-2\gamma-1/2}\ln^{\mu}(N)\right),& \mbox{$f(x)=(1-x)^{\gamma}\ln^{\mu}(1-x)g(x)$}\\
   {\cal O}(N^{m-s-1/2}\ln^{\mu}(N)),&\mbox{$f(x)=|x-z_0|^s\ln^{\mu}|z-z_0|g(x)$}\\
    {\cal O}\left(N^{m-\lambda-2\delta-1/2}\ln^{\mu}(N)\right),& \mbox{$f(x)=(1+x)^{\delta}\ln^{\mu}(1+x)g(x)$},\end{array}\right.
\end{equation}
\begin{equation}
\quad\quad \|f-\mathcal{P}_{N}^{f,Ch}\|_{ H^{m,\alpha,\beta}(\Omega)}=\left\{\begin{array}{ll}
   {\cal O}\left(N^{m-1/2-2\gamma}\ln^{\mu}(N)\right),& \mbox{$f(x)=(1-x)^{\gamma}\ln^{\mu}(1-x)g(x)$}\\
   {\cal O}(N^{m-s-1/2}\ln^{\mu}(N)),&\mbox{$f(x)=|x-z_0|^s\ln^{\mu}|z-z_0|g(x)$}\\
    {\cal O}\left(N^{m-1/2-2\delta}\ln^{\mu}(N)\right),& \mbox{$f(x)=(1+x)^{\delta}\ln^{\mu}(1+x)g(x)$}\end{array}\right.
\end{equation}
where $g\in C^{\infty}[-1,1]$, $z_{0}\in(-1,1)$, $\min\{\lambda+\gamma,\lambda+2\gamma,\lambda+ \delta,,\lambda+2 \delta\}>m-\frac{1}{2}$,  $s>m-\frac{1}{2}$ with $\lambda\ge 0$.  
In particular, if $\gamma,\, \delta$ are integers and $\mu\ge 1$, then 
\begin{equation}\label{eq:asytranjacexpanweights2}
\quad\,\,  \|f-\mathcal{P}_{N}^{f,Ge}\|_{H^{m,\alpha,\beta}(\Omega)}=\left\{\begin{array}{ll}
    {\cal O}(N^{m-\lambda-2\gamma-1/2}\ln^{\mu-1}(N)),& \mbox{$f(x)=(1-x)^{\gamma}\ln^{\mu}(1-x)g(x)$}\\
 {\cal O}\left(N^{m-\lambda-2\delta-1/2}\ln^{\mu-1}(N)\right),& \mbox{$f(x)=(1+x)^{\delta}\ln^{\mu}(1+x)g(x)$},\end{array}\right.
\end{equation}
\begin{equation}
\quad\quad \|f-\mathcal{P}_{N}^{f,Ch}\|_{H^{m,\alpha,\beta}(\Omega)}=\left\{\begin{array}{ll}
   {\cal O}\left(N^{m-1/2-2\gamma}\ln^{\mu-1}(N)\right),& \mbox{$f(x)=(1-x)^{\gamma}\ln^{\mu}(1-x)g(x)$}\\
    {\cal O}\left(N^{m-1/2-2\delta}\ln^{\mu-1}(N)\right),& \mbox{$f(x)=(1+x)^{\delta}\ln^{\mu}(1+x)g(x)$}.\end{array}\right.
\end{equation}
\end{corollary}

Numerical results for these estimates on the  boundary or an interior
point are illustrated in {\sc Fig.} \ref{fig:figs8} and {\sc Fig.} \ref{fig:figs9}, which indicates the optimal orders of the estimates .

\begin{figure}[hpbt]
\centerline{\includegraphics[height=5cm,width=16cm]{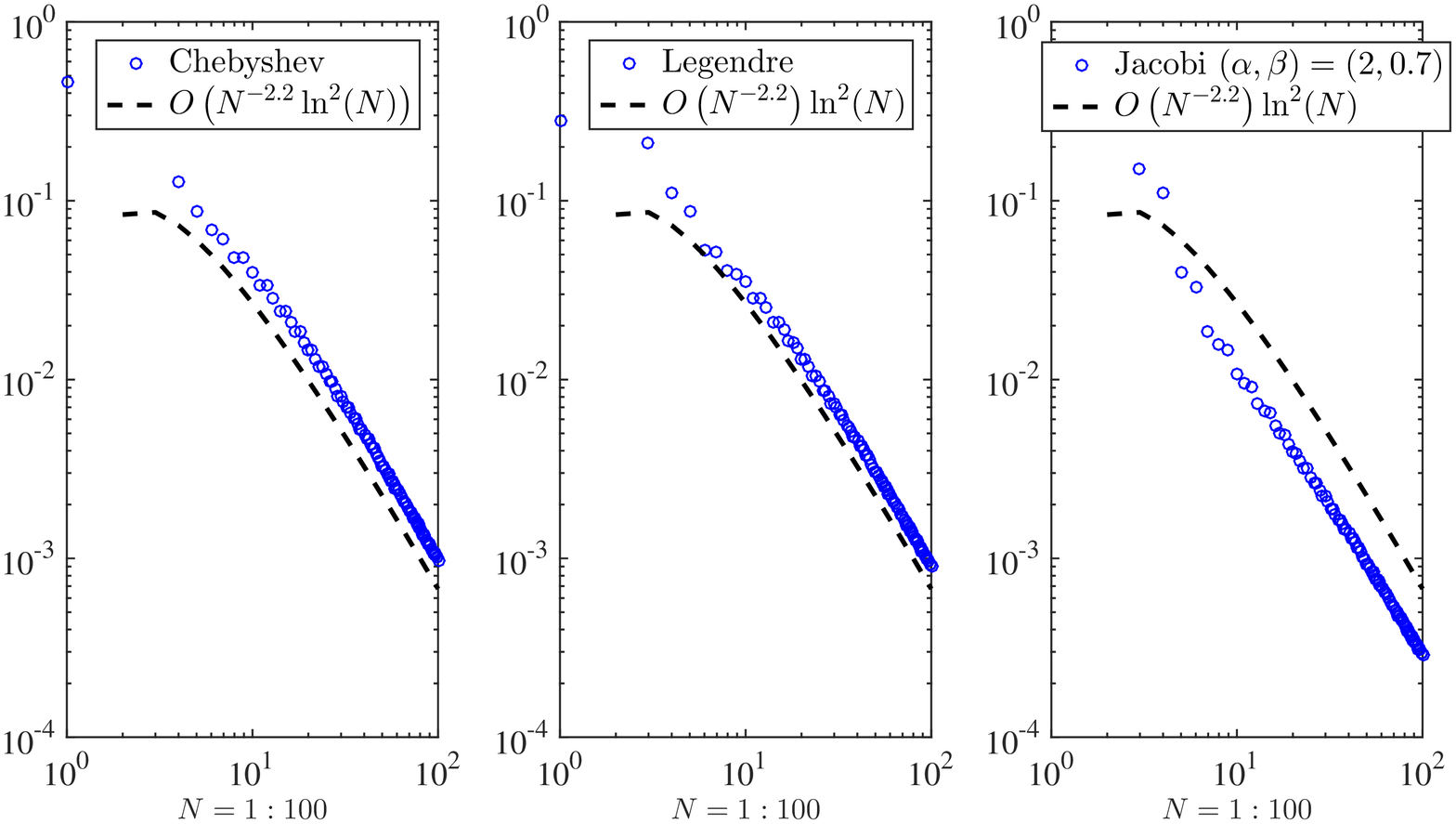}}
\caption{The weighted norm errors of the truncated Chebyshev, Legendre and Jacobi  expansions $\|f-\mathcal{P}_{n}^{f}\|_{H^{m,\alpha,\beta}(\Omega)}$ for $f(x)=|x-\frac{1}{2}|^{2.7}\ln^2|x-\frac{1}{2}|$, respectively.}\label{fig:figs8}
\end{figure}

\begin{figure}[hpbt]
\centerline{\includegraphics[height=6cm,width=16cm]{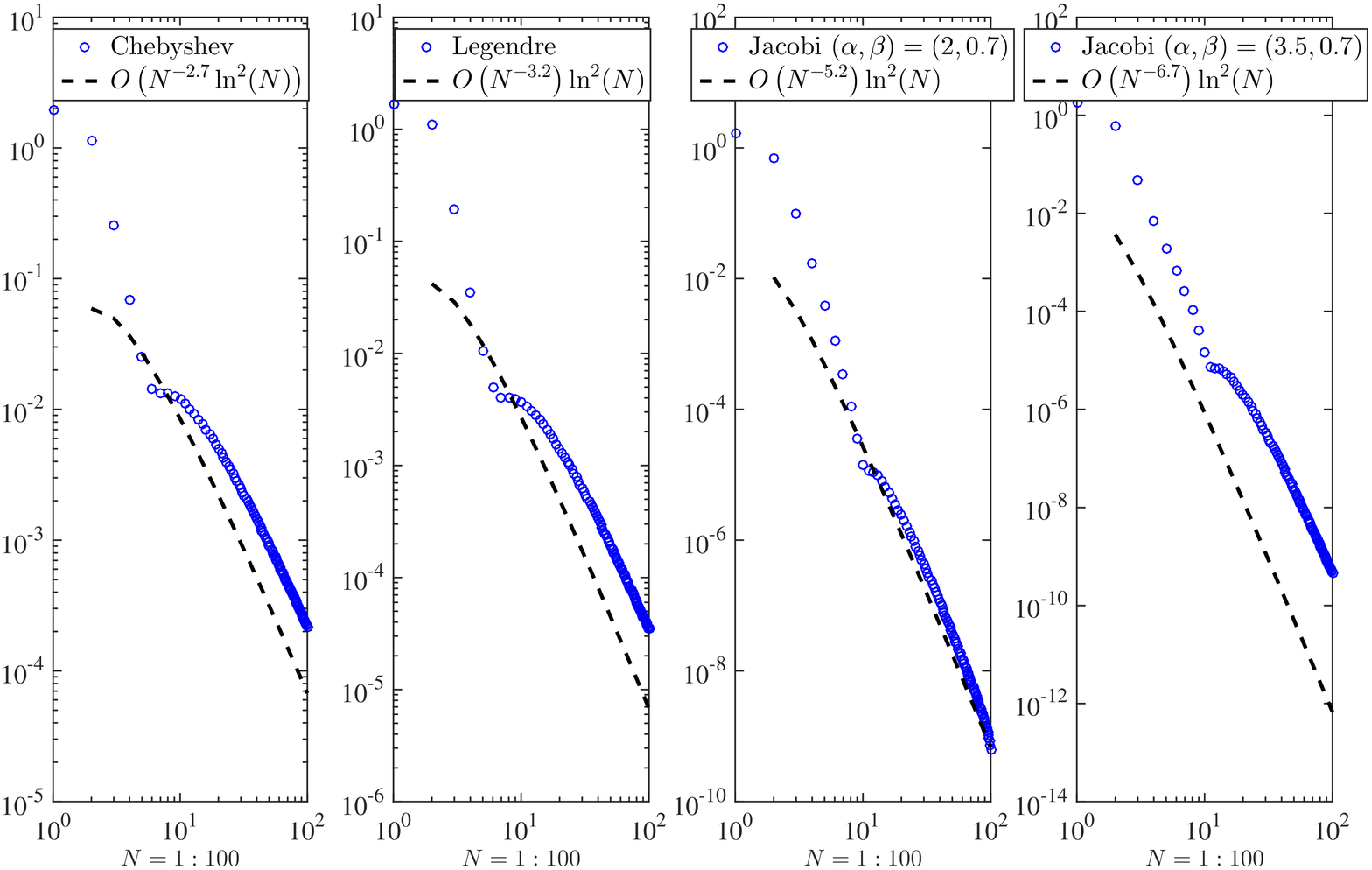}}
      \caption{The weighted norm errors of the truncated Chebyshev, Legendre and Jacobi ) expansions $\|f-\mathcal{P}_{n}^{f}\|_{H^{m,\alpha,\beta}(\Omega)}$ with $m=1$ for $f(x)=(1-x)^{1.6}\ln^2(1-x)$.}\label{fig:figs9}
\end{figure}

{\sc Remark 8}. In the case $\mu=0$ and $s$ even, all the three type projections $\|f-\mathcal{P}_{N}^{f}\|_{H^{m,\alpha,\beta}(\Omega)}$ are exponentially decayed \cite{Bernstein,Trefethen,Wang1,Wang2,Xiang2012,XWZ,ZWX}.

\section{Final remarks}
Applying the technique of the separation of singularities, the above results can be extended to the general functions with interior and boundary singularities for
 \begin{equation}\label{gen}{\displaystyle
f(x)=g(x)\prod_{i=1}^m|x-x_i|^{\gamma_i}\ln^{\mu}|x-x_i|=\sum_{i=1}^m|x-x_i|^{\gamma_i}\ln^{\mu}|x-x_i|g_i(x)}
\end{equation}
where $-1=x_m<x_{m-1}<\cdots<x_{2}<x_1=1$, $g_i\in C^{\infty}[-1,1]$, $\gamma_i\ge  0$ for $i=1,m$, $\gamma_i>0$  for $i=2,\ldots,m-1$ and $\mu$ is a positive integer \cite[pp. 219-220]{Tuan}.

\begin{corollary}\label{cor:cors2}
 Suppose $f(x)$ is defined by (\ref{gen}), then the coefficients in the Jacobi series of $f(x)$ satisfy
  \begin{equation}
   |a_{n}(\alpha,\beta)|=\left\{\begin{array}{l}
      {\cal O}\left(\min\left\{\ln^{\mu-1}(n)n^{-\min\left\{1+\alpha+2\gamma_0,1+\alpha+2\gamma_m\right\}},n^{-\min\left\{\gamma_2+\frac{1}{2},\ldots,\gamma_{m-1}+\frac{1}{2}\right\}}\ln^{\mu}(n)\right\}\right),\\
      \quad\quad\quad\gamma_{1},\gamma_{m}\in{\cal N}_0,\\
       {\cal O}\left(\min\left\{\ln^{\mu-1}(n)n^{-1-\alpha-2\gamma_1},n^{-\min\left\{1+\alpha+2\gamma_m,\gamma_2+\frac{1}{2},\ldots,\gamma_{m-1}+\frac{1}{2}\right\}}\ln^{\mu}(n)\right\}\right),\\
       \quad\quad\quad\gamma_{1}\in {\cal N}_0,\,\gamma_{m}\notin{\cal N}_0,\\
 {\cal O}\left(\min\left\{\ln^{\mu-1}(n)n^{-1-\beta-2\gamma_m},n^{-\min\left\{1+\alpha+2\gamma_1,\gamma_2+\frac{1}{2},\ldots,\gamma_{m-1}+\frac{1}{2}\right\}}\ln^{\mu}(n)\right\}\right),\\
 \quad\quad\quad\gamma_{1}\notin {\cal N}_0,\,\gamma_{m}\in{\cal N}_0,\\
      {\cal O}\left(n^{-\min\left\{1+\alpha+2\gamma_1,1+\beta+2\gamma_m,\gamma_2+\frac{1}{2},\ldots,\gamma_{m-1}+\frac{1}{2}\right\}}\ln^{\mu}(n)\right),\\
      \quad\quad\quad\gamma_{1},\gamma_{m}\notin{\cal N}_0
      \end{array}\right.
\end{equation}
\end{corollary}

From (\ref{eq:gegcoeffs}) and (\ref{eq:checoeffs}), one may obtain sharp bounds for the Gegenbauer coefficients (\ref{eq:gegexpan}) and Chebyshev  coefficients (\ref{eq:chebexpan}), 
and optimal convergence rates 
for suitably differentiable functions in a similar way.

\vspace{0.36cm}\noindent {\bf Acknowledgments:}  
The author
 thanks Prof. Li-Lian Wang,  Dr. Guidong Liu, Desong Kong,  Qingyang Zhang and Jiangli Liang for  many
constructive discussion and comments that helped to improve the presentation of
this paper.

\end{document}